\documentclass[12pt,a4paper]{amsart}
\usepackage{amsrefs}
\usepackage[T1]{fontenc}
\usepackage[latin1]{inputenc}
\usepackage[body={15cm, 24cm}]{geometry}
\usepackage{graphicx,mathrsfs,amssymb}

\newtheorem{theorem}{\textbf{Theorem}}[section]
\newtheorem{proposition}[theorem]{\textbf{Proposition}}
\newtheorem{lemma}[theorem]{\textbf{Lemma}}

\theoremstyle{definition}
\newtheorem{definition}[theorem]{\textbf{Definition}}

\theoremstyle{remark}

\newtheorem{example}[theorem]{Example}
\newtheorem{remark}[theorem]{Remark}

\title[time optimal control with differential inclusions]{Optimality conditions and regularity results\\ for time optimal control problems\\ with differential inclusions}

\author{Piermarco Cannarsa}
\address{\hspace{-0.5em}\begin{tabular}{ll}Piermarco Cannarsa:&Dipartimento di Matematica,\\& Universit\`a di Roma 'Tor Vergata'\\& 
Via della Ricerca Scientica 1, I-00133 Roma, Italy.\end{tabular}}
\email{cannarsa@mat.uniroma2.it}

\author{Antonio Marigonda}
\address{\hspace{-0.5em}\begin{tabular}{ll}Antonio Marigonda:&Department of Computer Science,\\& University of Verona\\ &Strada Le Grazie 15, I-37134 Verona, Italy.\end{tabular}}
\email{antonio.marigonda@univr.it}

\author{Khai T. Nguyen}
\address{\hspace{-0.5em}\begin{tabular}{ll}Khai T. Nguyen:&Department of Mathematics,\\& Penn State University,\\ &University Park, Pa. 16802, U.S.A.\end{tabular}}
\email{ktn2@psu.edu}

\date{\today}



\thanks{The authors wish to acknowledge the support obtained by the ITN - Marie Curie Grant n.~264735-SADCO, and by the CNRS and  INdAM through the GDRE CONEDP. The third author also wishes to acknowledge the support obtained by the ERC Starting Grant 2009 n.240385 ConLaws in Padova.}
\begin{document}

\begin{abstract}
We study the time optimal control problem with a general target $\mathcal S$ for a class of differential inclusions that satisfy mild smoothness and controllability assumptions. In particular, we do not require Petrov's condition at the boundary of $\mathcal S$. 
Consequently, the minimum time function $T(\cdot)$ fails to be locally Lipschitz---never mind semiconcave---near $\mathcal S$. Instead of such a regularity, we use an exterior sphere condition for the hypograph of $T(\cdot)$  to develop the analysis.  In this way, we obtain dual arc inclusions which we apply to show the constancy of the Hamiltonian along optimal trajectories and other  optimality conditions in Hamiltonian form. We also prove an upper bound for the Hausdorff measure of the  set of all nonlipschitz points of $T(\cdot)$ which implies that  the minimum time function is  of special bounded variation.
\end{abstract}

\maketitle

\section{Introduction}

In this paper we study the time optimal control problem for the differential inclusion
\begin{equation}\label{eq:mfun}\begin{cases}\dot x(t)\in F(x(t)),\textrm{ for a.e. }t>0,
%\vspace{.2cm}
\\ x(0)=x_0\in\mathbb R^n,\end{cases}\end{equation}
with a given closed target set $\mathcal S\subseteq\mathbb R^n$. The dynamics is described by
a set-valued Lipschitz continuous function $F:\mathbb R^n\rightrightarrows\mathbb R^n$, whose values are assumed to
be compact, convex and nonempty.

\medskip

For each trajectory of \eqref{eq:mfun} starting from $x_0$, i.e., for each absolutely continuous function $y^{x_0}(\cdot)\in AC([0,+\infty[;\mathbb R^n)$
satisfying $\dot y^{x_0}(t)\in F(y^{x_0}(t))$ for a.e. $t>0$ and $y^{x_0}(0)=x_0$, we set
\[\theta(y^{x_0}(\cdot)):=\inf\{t\geq 0:\,y^{x_0}(t)\in \mathcal S\},\]
with the convention $\inf\emptyset=+\infty$. The \emph{minimum time function} is defined by
\[T(x_0):=\inf\left\{\theta(y^{x_0}(\cdot)):\, y^{x_0}(\cdot)\textrm{ is a trajectory of \eqref{eq:mfun} with }y^{x_0}(0)=x_0\right\}.\]

\medskip

When finite, $T(x_0)$ represents the minimum time needed to steer the point $x_0$ to the target $\mathcal S$ following the trajectories of \eqref{eq:mfun}.
The study of the regularity of $T(\cdot)$ is a central problem in optimal control theory, and the literature offers a huge choice of  papers on such a topic.
In particular, starting from the seminal paper \cite{CS0}, the regularity property of $T(\cdot)$ called \emph{semiconcavity} was extensively studied and used to deduce new optimality conditions.
Roughly speaking, semiconcavity amounts to the local Lipschitz continuity of $T(\cdot)$ plus a uniform exterior sphere condition for the hypograph of $T(\cdot)$.
Using these geometric properties, one can derive further regularity results for $T(\cdot)$, such as $BV$ estimates for $T(\cdot)$ and $\nabla T(\cdot)$,
 the existence of Taylor's expansion of order two around a. e. point, and bounds for the singular sets of $T$. We refer to \cite{CS} for an introduction to semiconcave functions.

\medskip

Semiconcavity (and semiconvexity) results, however, rely essentially on a strong controllability assumption---the so-called Petrov condition---which is actually equivalent to the local Lipschitz continuity of $T$.
When this assumption is removed, $T(\cdot)$ may fail to be semiconcave, but it still retains the external sphere property of the hypograph (or epigraph).
In \cite{CMW}, for parameterized linear multifunctions, the epigraph of $T(\cdot)$ was proved to satisfy a strong external sphere condition, 
which is called \emph{positive reach property}
in the sense of Federer (also known, in literature, as \emph{$\varphi$-convexity} or \emph{proximal smoothness}). The regularity properties of the class
of functions whose epigraph has positive reach turned out to be comparable with the properties of semiconcave/semiconvex functions, we refer to \cite{CM} for details.

\medskip

The above geometric approach was  also generalized to the nonlinear case: in \cite{CKH}, the positive reach property of the hypograph of $T(\cdot)$ was obtained replacing Petrov's condition by the  assumption
that the normal cone to the hypograph is pointed (i.e., contains no lines). Without such an  assumption, in \cite{NK}, the hypograph of $T(\cdot)$ was proved to
satisfy a weaker external sphere condition. In the same paper, it was shown that such a condition yields essentially the same regularity properties as in the case when
the hypograph has positive reach.

\medskip

The study of the singular set for $T(\cdot)$, beyond the semiconcave case, was performed  in \cite{CM3} for the positive reach case, in \cite{NV} for the weak external sphere
property, and in \cite{MKV} for a further generalized case. In all these papers the structure of the singular set was analyzed, providing upper bounds for the Hausdorff 
measure of such a set. Lower estimates were obtained in \cite{ACK} and \cite{CVL}.

\medskip

So far, all the results in the above papers were proved for a state equation given in the form of  a smooth $C^{1,1}$-parameterized multifunction. The smoothness of the parameterization
was a crucial issue to gain semiconcavity estimates, and was extensively used  in all its generalizations as well. Note that, although a Lipschitz multifunction always
posseses a Lipschitz parameterization (see, for instance, \cite{AF}),  it is still an open problem to find conditions for such a parameterization to be smooth.

\medskip

This motivates a separate study of the minimum time problem for differential inclusions. The first result in this direction is due to \cite{CaW}, where the
semiconcavity of the value function of the Mayer problem for system \eqref{eq:mfun} was proved. Instead of searching for a smooth parametrization of $F(\cdot)$, the main idea of  \cite{CaW} is to use the smoothness of the Hamiltonian 
associated with system \eqref{eq:mfun}, defined by
\begin{equation}\label{eq:ham}H(x,p)=\sup_{v\in F(x)}\langle p,v\rangle,\hspace{2cm}(x,p)\in\mathbb R^n\times\mathbb R^n\,.\end{equation}

\medskip

This approach was extended in \cite{CaMW} to the study of the minimum time function, still in the semiconcavity framework. A further step to extend this kind of analysis
beyond semiconcavity, was performed in \cite{CK}, which can be seen as the counterpart of \cite{CKH} and \cite{NK} for differential inclusions, proving the 
external sphere property enjoyed by $\mathrm{hypo}\,T$ under milder controllability assumptions.

\medskip

The main tool used in our analysis is the \emph{dual arc inclusion} (see \cite{CaMW}), i.e., the natural nonsmooth counterpart 
of the Hamiltonian system associated with the problem, that is,
\begin{equation}\label{eq:Couple-S-intro}
\begin{cases}
\hspace{.3cm}
\dot{x}(t)=\nabla_p H(x(t),p(t)),\vspace{.2cm}\\ 
-\dot{p}(t)\in\partial_xH(x(t),p(t)).
\end{cases}
\end{equation}
If the trajectory $x(\cdot)$ starting from $x_0$ at $t=0$ is optimal, then the above system admits a solution when coupled
with a terminal condition for $p(T(x_0))$ that turns out to be a proximal inner normal to the target set $\mathcal S$ at $\bar x=x(T(x_0))$.

\medskip

In this paper, we prove two results for system \eqref{eq:Couple-S-intro}: in the first one, we state that the superdifferentiability of the minimum time function $T(\cdot)$
along an optimal trajectory is completely determined by the value of the Hamiltonian at the endpoint $(x(T(x_0)),p(T(x_0)))$.
More precisely, if we are able to find a nontrivial proximal inner normal to the target $\mathcal S$ at the endpoint $\bar x=x(T(x_0))$ such that the value of
the Hamiltonian computed at this inner normal at $\bar x$ is nonzero,
then the $p$-part of the solution of \eqref{eq:Couple-S-intro} yields a continuous selection of $\partial^P T$ along the trajectory $x(\cdot)$.
In particular, we have that $\partial^PT(x(t))\neq \varnothing$ for all $t\in[0,T(x_0)]$.

\medskip

Due to the presence of horizontal normals to $\mathrm{hypo}\,T$, the nonsmoothness of $T(\cdot)$ propagates similarly along the flow of an optimal trajectory: indeed
we show that, if the value of
the Hamiltonian computed at this inner normal at $\bar x$ is zero, then the same $p$-part of the solution yields a continuous selection
of $\partial^{\infty} T$ along the optimal trajectory. So, once again, $\partial^\infty T(x(t))\neq \varnothing$ for all $t\in[0,T(x_0)]$.

\medskip

The second  result of this paper, which is a consequence of the above propagation of superdifferentiability, is concerned with the constancy
of the Hamiltonian along optimal trajectories, which somewhat resembles the classical case. In order to prove this result, we use the fact that the value of the
Hamiltonian, at the terminal point of an optimal trajectory, computed at a suitable inner normal, yields the existence of a supergradient of $T$,
or of an horizontal supergradient of $T$, at \emph{all} points of the trajectory. This fact forces the value of the Hamiltonian, along the solution of the dual arc inclusion, 
to be identically $1$ in the former cases, and $0$ in the latter. 

\medskip

A partial converse of the above result, i.e., a sufficient condition for optimality, is also established in the case of superdifferentiability.
We show that, if we have a solution of the generalized characteristic system
\[\begin{cases}\hspace{.3cm}\dot x(t)=\nabla_p H(x(t),p(t)),\vspace{.2cm}\\  - p(t)\in \partial^PT(x(t)),\end{cases}\]
along which the Hamiltonian is constantly equal to $1$, then the $x$-part of the solution is actually an optimal trajectory.

\medskip

The last part of the paper is devoted to the study of $SBV$ regularity of the minimum time function, which guarantees that the singular part of the distributional gradient of  $T(\cdot)$ has no Cantor component---a property that has several applications to the calculus of variations (see, e.g., \cite{AFP}). 
Actually, the result we prove is slightly stronger than just establishing the $SBV$ regularity of $T(\cdot)$. Indeed, Theorem~\ref{thm:sbv} ensures that, if the subset of $\partial\mathcal S$ at which 
$H$ vanishes for a nontrivial choice of inner normals is \emph{small enough}, then there are \emph{few}
optimal trajectories along which the normal cone to $\mathrm{hypo}\,T$ has an horizontal part. So, at \emph{most} points, $\partial^PT(x)$ 
is bounded and $T(\cdot)$ turns out to exhibit a Lipschitz behaviour.

\medskip

The paper is structured as follows: in Section \ref{sec:prel} we fix the notation, and recall definitions and  preliminaries from nonsmooth analysis, expecially concerning differential inclusions. 
In Section \ref{sec:stand}, we discuss our standing assumptions and recall some useful results from \cite{CaW}, while Section \ref{sec:main} is devoted to the main results of the paper and the analysis of some of their consequences.  

\section{Preliminaries and notation}\label{sec:prel}

Our ground space will be the Euclidean space $\mathbb R^n$.

\begin{definition}
Let $ \Omega,K,$ and $S$ be, respectively, an open, closed, and any subset of $\mathbb R^n$, let $x=(x_1,\dots,x_n)$ and $y=(y_1,\dots,y_n)$ be points of $\mathbb R^n$, let $r>0$, and let $k\geq 0$ be an integer.
We denote by:

{\allowdisplaybreaks
\begin{tabular}{rll}
$\langle x,y\rangle$&$:=\displaystyle \sum_{i=1}^n x_iy_i$&the \emph{scalar product} in $\mathbb R^n$;\\
$\|x\|$&$:=\displaystyle \sqrt{\langle x,x\rangle}$&the \emph{Eulidean norm} in $\mathbb R^n$;\\
&$\partial S,\,\mathrm{int}(S),\,\overline{S}$&the \emph{topological boundary},\\&& \emph{interior} and \emph{closure} of $S$;\\
$\mathrm{diam}(S)$&$:=\sup\{\|z_1-z_2\|:\,z_1,z_2\in S\}$&the \emph{diameter} of $S$;\\
$\mathcal P(S)$&$:=\{B\subseteq\mathbb R^n:\,B\subseteq S\}$&the \emph{power set} of $S$;\\
$\mathbb B^n$&:=$\{w\in\mathbb R^n:\,\|w\|< 1\}$&the \emph{unit open ball}\\&& (centered at the origin);\\
$\mathbb S^{n-1}$&$:=\{w\in\mathbb R^n:\,\|w\|=1\}=\partial\mathbb B^n$&the \emph{unit sphere}\\&& (centered at the origin);\\
$B(y,r)$&$:=\{z\in\mathbb R^n:\,\|z-y\|<r\}=y+r\mathbb B^n$&the \emph{open ball} centered\\&& at $y$ of radius $r$;\\
$d_K(y)$&$:=\mathrm{dist}(y,K)=\min\{\|z-y\|:\,z\in K\}$&the \emph{distance} of $y$ from $K$;\\
$\pi_K(y)$&$:=\{z\in K:\, \|z-y\|=d_K(y)\}$&the \emph{set of projections} of $y$\\&& onto $K$;\\
$S^c$&$:=\mathbb R^n\setminus S$&the complement  of $S$.
\end{tabular}
}

\medskip
Also, $C^k(\Omega)$ stands for the space of all  $f:\Omega\to\mathbb R$ with continuous derivatives up to order $k$,
$C^k_b(\Omega)$ collects the functions of $C^k(\Omega)$ with bounded derivatives of order $k$,
and $C^k_c(\Omega)$  the ones with compact support.

\medskip
If $\pi_K(y)=\{\xi\}$, i.e. it is a singleton, we will identify the set $\pi_K(y)$ with its unique element and write $\pi_K(y)=\xi$.
The \emph{characteristic function} $\chi_S:\mathbb R^d\to\{0,1\}$ of $S$ is defined as $\chi_S(x)=1$ if $x\in S$ and $\chi_S(x)=0$ if $x\notin S$.
\end{definition}

\begin{definition}
Let $X$ be a vector space. A set $C\subseteq X$ is \emph{convex} if for every $x_1,x_2\in C$, $\lambda\in[0,1]$, we have that $\lambda x_1+(1-\lambda)x_2\in C$.
If $S\subseteq X$ is a set, the smallest (with respect to inclusion) convex set which contains $S$ is called the \emph{convex hull} of $S$ and it is defined as
\[\mathrm{co}\,S:=\bigcap_{C\in \mathscr F(S)}C,\]
where $\mathscr F(S):=\{C\subseteq X:\,C\supseteq S\textrm{ and }C\textrm{ is convex}\}$.
We have that $S=\mathrm{co}\,S$ iff $S$ is convex.
\end{definition}

\begin{definition}
Let $X$ be a vector space and $f:X\to \mathbb R\cup\{\pm\infty\}$ be a function. We recall the definitions of
\begin{align*}
\mathrm{dom}\,f&:=\{x\in X:\, f(x)\in\mathbb R\},\textrm{ the \emph{domain} of $f$};\\
\mathrm{epi}\,f&:=\{(x,\beta)\in X\times\mathbb R:\, x\in\mathrm{dom}\,f,\, \beta\ge f(x)\},\textrm{ the \emph{epigraph} of $f$};\\
\mathrm{hypo}\,f&:=\{(x,\alpha)\in X\times\mathbb R:\, x\in\mathrm{dom}\,f,\, \alpha\le f(x)\},\textrm{ the \emph{hypograph} of $f$}.
\end{align*}
If $X$ is a topological vector space, we say that $f$ is \emph{lower semicontinuous}
(shortly: l.s.c.) if $\mathrm{epi}\,f$ is closed in $X\times\mathbb R$ with respect to the product topology on $X\times \mathbb R$,
i.e. \[\liminf_{y\to x}f(y)\ge f(x).\]
A function $g$ is called \emph{upper semicontinuous} (shortly u.s.c.) if $-g$ is l.s.c.
\end{definition}

\begin{definition}[Lipschitz functions]
Given the Banach spaces $X,Y$, and two open sets $U\subseteq X$, $V \subseteq Y$, a function $f : U \to V$ is said to be a \emph{Lipschitz continuous function} 
($f\in \mathrm{Lip}(U)$) if there exists $C>0$, called a \emph{Lipschitz constant}, such that for every $x_1,x_2\in U$
\[\|f(x_1)-f(x_2)\|_Y \leq C \|x_1-x_2\|_X.\] 
A function $f:U\to V$ is called \emph{locally Lipschitz continuous} ($f\in \mathrm{Lip}_{\mathrm{loc}}(U)$) if it is Lipschitz continuous on every compact subset of $U$. 
Given $x\in X$, we say that $f:X\to Y$ is Lipschitz continuous at $x$ if there exists a neighborhood $U$ of $x$ in $X$ such that $f:U\to Y$ is Lipschitz continuous.
\end{definition}

We recall the following classical result on regularity of Lipschitz functions holding for spaces of finite dimensions:

\begin{theorem}[Rademacher's Theorem]
Let $X$ be a Banach space of finite dimension, let $U\subseteq X$ be open, 
and let $f: U \to \mathbb R$ be a locally Lipschitz function. 
Then $f$ is differentiable almost everywhere with respect to Lebesgue measure.
\end{theorem}

\begin{definition}[$BV$ functions]
Let $\Omega$ be an open subset of $\mathbb R^n$, $u\in L^1(\Omega)$. We say that $u$ is 
a \emph{function of bounded variation} in $\Omega$ and write $u\in BV(\Omega)$ 
if there exist Radon measures $\mu_1,\dots,\mu_n$ on $\Omega$ such that
\[\int_{\Omega} u(x) \partial_{x_i}\varphi(x)\,dx=-\int_{\Omega} \varphi(x) d\mu_i(x),\hspace{1cm}\textrm{ for all }\varphi\in C^1_c(\Omega).\]
The vector-valued measure $Du:=(\mu_1,\dots,\mu_n)$ is called the distributional (full) gradient of $u$.
We say that $u\in BV_{\mathrm{loc}}(\Omega)$ if $u\in BV(\Omega')$ for every bounded open subset $\Omega'$ of $\Omega$.
\end{definition}
We now recall a few notions related to functions of bounded variation, referring  the reader  to  \cite{AFP} for more details.  If $u\in BV$, then $Du$ can be decomposed into an {\em absolutely continuous part}  w.r.t. the Lebesgue measure, herafter denoted by $\nabla u\, dx$ or   $\nabla u$, and a {\em singular part}, $D^su$. One can further introduce the {\em jump part of}  $D^su$, $Ju$, taking the restriction of $D^su$ to the set of all points of $\Omega$ at which $u$ has no approximate limit. Then, the difference $D^su-Ju$ is the so-call {\em Cantor part of} $D^su$, labeled $D^cu$.
Finally, we say that $u\in BV_{\mathrm{loc}}(\Omega)$ is a 
special function of locally bounded variation, and we write $u\in SBV_{\mathrm{loc}}(\Omega)$, if the Cantor part  $D^cu$ vanishes. 
%The following result  is needed in Section 4.
\begin{proposition}[Proposition~4.2 in \cite{AFP}] \label{propSBV}  Let $u\in BV(\Omega)$. Then $u \in SBV(\Omega)$ if 
and only if $D^su$ is concentrated on a Borel set of $\sigma$-finite  $\mathscr{H}^{n-1}$ measure.
\end{proposition}
Thus, $u \in SBV(\Omega)$ if it vanishes outside a countably $\mathscr{H}^{n-1}$-rectifiable set.
%\begin{definition}[$SBV$ functions]
%A function $u\in BV(\Omega)$ is said to be of \emph{special bounded variation} if  the singular part $D^su$
%of its gradient is concentrated on a set of $\sigma$-finite $\mathscr H^{n-1}$-measure, where $\mathscr H^k$ denotes the $k$-dimensional Hausdorff measure. In this case, we write  $u\in SBV(\Omega)$.
%\end{definition}
\begin{definition}[Proximal normals]
Let $S$ be a closed subset of $\mathbb R^n$. 
A vector $v$ is called a \emph{proximal normal} to $S$ at $x \in S$  
if there exists $\sigma =\sigma(v,x)\geq 0$ such that 
\begin{equation}\label{eq:proxnor}
\langle v, y-x\rangle \leq \sigma \|y-x\|^2,\textrm{ for every $y \in S$}.
\end{equation}
The set of all proximal normals to $S$ at $x$ will be denoted by $N_S^P(x)$.

\medskip
Given $\rho>0$, we say that $v\in N_{S}^{P}(x)$ is \emph{realized by a ball with radius} $\rho$ if one can take $\sigma=\frac{\|v\|}{2\rho}$ in \eqref{eq:proxnor}. 
\end{definition}
Observe that  $x\in N_{Q}^{P}(x)$ if and only   $d_Q(x+\lambda v)=\lambda\|v\|$ for some $\lambda>0$.
Moreover, if $v\in N_{S}^{P}(x)$ is realized by a ball with radius $\rho$, then $B(x+\rho v,\rho)\cap S=\emptyset$.
\begin{remark}
We recall that, when $S$ is convex, we can take $\sigma=0$ in the above definition.
Hence, the proximal normal cone at $x$ reduces to the \emph{normal cone in the sense of convex analysis},
namely the set of vectors $v\in\mathbb R^n$ such that $\langle v, y-x \rangle \leq 0$ for all $y \in S$.
\end{remark}

\begin{definition}[External sphere condition]
Let $S\subset\mathbb{R}^n$ be closed and let $\theta:\partial S\rightarrow ]0,\infty[$ be continuous. 
We say that $S$ satisfies the \emph{$\theta$-external sphere condition} if for every $x\in\partial S$, 
there exists a nonzero vector $v\in N^P_Q(x)$ which is realized by a ball of radius $\theta(x)$.

We say that a closed subset $Q$ of $\mathbb R^n$ satisfies the \emph{$\theta$-internal sphere condition}
if and only if $\overline{\mathbb R^n\setminus Q}$ satisfies the $\theta$-external sphere condition.
\end{definition}

\begin{definition}[Proximal supergradients]
Let $f:\mathbb{R}^n\to [-\infty,+\infty]$ be an upper semicontinuous function.
For every $x\in\mathrm{dom}(f)$, we denote by 
\begin{equation}
\partial^Pf(x)\doteq\left\{\xi\in\mathbb{R}^n:\,(-\xi,1)\in N^P_{\mathrm{hypo}(f)}(x,f(x))\right\},
\end{equation}
the set of \emph{proximal supergradients} of $f$ at $x$, and by  
\begin{equation}
\partial^{\infty}f(x)\doteq\left\{\xi\in\mathbb{R}^n:\,(-\xi,0)\in N^P_{\mathrm{hypo}(f)}(x,f(x))\right\}.
\end{equation}
the set of \emph{horizontal proximal supergradients} of $f$ at $x$.
\end{definition}
We recall that, if $f$ is Lipschitz continuous at $x$, then $\partial^{\infty}f(x)$ is empty. The converse in general is false.
\begin{definition}[Support function]
Let $C\subseteq\mathbb R^n$ be nonempty. The \emph{support function} $\sigma_C:\mathbb R^n\to]0,+\infty]$ is defined as
$\sigma_C(p)=\displaystyle\sup_{v\in C}\langle p,v\rangle$. 
\end{definition}
It can be proved that $\sigma_C(p)=\sigma_{\overline{\mathrm{co}\,C}}(p)$.

Let $\Omega$ be an open subset of $\mathbb R^n$, let $x_0\in \Omega$, and let $f:\Omega\to\mathbb R$ be Lipschitz continuous at $x_0$.
\begin{definition}[Clarke's generalized gradient]
We define \emph{Clarke's generalized gradient of $f$ at $x_0$} by setting
\[\partial f(x)=\mathrm{co}\left\{v\in\mathbb R^n:\, \exists\{y_i\}_{i\in\mathbb N}\subset\Omega\textrm{ s.t. }f\textrm{ is differentiable at }y_i,\,\begin{array}{l}y_i\to x_0\\\nabla f(y_i)\to v\end{array}\right\}.\]
\end{definition}

We now recall  the definition and some properties of semiconcave functions, referring the reader to \cite{CS} for further properties and characterizations. Let $\Omega$ be an open subset of $\mathbb R^n$ and let $c\in\mathbb R$. 
\begin{definition}[Semiconcave functions]
We say that  $f:\Omega\to\mathbb R$ is \emph{semiconcave in $\Omega$ with 
semiconcavity constant $c_0>0$} if $f$ is continuous and, for every $x_1,x_2\in\Omega$ such that the line segment $\{\lambda x_1+(1-\lambda)x_2:\,\lambda\in[0,1]\}$ is contained in $\Omega$,
we have
\[\dfrac{f(x_1)+f(x_2)}2-f\left(\dfrac{x_1+x_2}2\right)\le \dfrac{c_0}4\|x_1-x_2\|^2.\]
In this case we will write $f\in SC(\Omega)$. We say that $f$ is \emph{locally semiconcave in $\Omega$}, and write $f\in LSC(\Omega)$, if $f\in SC(A)$
for every $A\subset \Omega$, $A$ open and bounded.
We say that $g:\Omega\to\mathbb R$ is semiconvex iff $-g$ is semiconcave.
\end{definition}

\begin{proposition}[Properties of semiconcave functions]
Let $\Omega$ be an open subset of $\mathbb R^n$, let $c\in\mathbb R$, and let $f:\Omega\to\mathbb R$ be continuous.
Then the following are equivalent:
\begin{enumerate}
\item $f\in SC(\Omega)$ with semiconcavity constant $c$;
\item $f(y)-f(x)\le \langle \xi,y-x\rangle +c\|y-x\|^2$ for all $\xi \in \partial^P f(x)$ and $y\in\Omega$. 
\end{enumerate}
\end{proposition}

\begin{definition}[Multifunctions]
Let $\Omega$ be a subset of $\mathbb R^n$. A map $F:\Omega\to \mathcal P(\mathbb R^m)$ will be called a \emph{multifunction}
or \emph{set-valued function}, and we will write $F:\Omega\rightrightarrows \mathbb R^m$.

We say that a multifunction $F:\Omega\rightrightarrows \mathbb R^m$ is:
\begin{enumerate}
\item[-] \emph{measurable}, if for all closed $C\subset\mathbb R^m$ we have that $\{x\in\Omega:\, F(x)\cap C\ne\emptyset\}$ is Lebesgue measurable;
\item[-] \emph{upper semicontinuous at $x_0\in\Omega$}, if for any open set $M$ containing $F(x_0)$ there exists an open set $A\subset\Omega$ with $x_0\in A$ such that
\[F(A):=\bigcup_{a\in A}F(a)\subseteq M;\]
\item[-] \emph{lower semicontinuous at $x_0\in\Omega$}, if for any $y_0\in F(x_0)$ and any open set $M$ containing $y_0$ there exists an open set $A\subset\Omega$ with $x_0\in A$ such that
$F(x)\cap M\ne\emptyset$;
\item[-] \emph{continuous at $x_0\in\Omega$}, if $F$ is both lower and upper semicontinuous at $x_0$;
\item[-] \emph{Lipschitz continuous}, if there exists $K>0$, called a Lipschitz constant of $F$, such that for any $x_1,x_2\in\Omega$ we have $F(x_2)\subseteq F(x_1)+K\|x_1-x_2\|\mathbb B^m$.
\end{enumerate}
\end{definition}

\begin{definition}[Differential inclusions]
Let $F:\mathbb R^n\rightrightarrows \mathbb R^n$ be a measurable multifunction. A \emph{solution} or \emph{trajectory} of the 
\emph{differential inclusion}
\begin{equation}\label{eq:diffinc}
\begin{cases}
\dot x(t)\in F(x(t)),\\
x(t_0)=x_0\in\mathbb R^n,
\end{cases} 
\end{equation}
is an absolutely continuous map $y^{x_0}:I\to\mathbb R^n$ such that $I$ is an open interval containing $t_0$, $y^{x_0}(t_0)=x_0$, and $\dot y^{x_0}(t)\in F(t,x(t))$
for a.e. $t\in I$.
\end{definition}
We refer the reader to the monograph \cite{AC} for further properties and results on multifunctions and differential inclusions. 

\begin{definition}[Minimum time function]\label{def:mintime}
Let $F:\mathbb R^n\rightrightarrows \mathbb R^n$ be a measurable multifunction,
and let $\mathcal S$ be a given closed subset of $\mathbb R^n$.
We define the \emph{minimum time function} $T:\mathbb R^n\to]0,+\infty]$ by setting:
\begin{equation}\label{eq:mintime}
T(x)=\inf\left\{\tau>0:\, \exists y^{x}(\cdot)\textrm{ satisfying }\dot x(t)\in F(x(t)),\;x(0)=x, \;y^x(\tau)\in S\right\}
\end{equation}
with the usual  convention $\inf\emptyset=+\infty$.
\end{definition}
The \emph{reachable} (or \emph{controllable}) set is defined as $\mathcal R=\mathrm{dom}(T)$.

\section{Standing hypothesis and first consequences}\label{sec:stand}
From now on, we will consider $t_0=0$ in the differential inclusion \eqref{eq:diffinc}, 
and the associated minimum time function $T$ defined in \eqref{eq:mintime}.

\begin{definition}[Hamiltonian]
The \emph{Hamiltonian} $H:\mathbb R^n\times\mathbb R^n\to\mathbb R$ associated with $F$ is given by
\begin{equation}\label{eq:Hamiltonian}H(x,p)=\ \sup_{v\in F(x)}\langle p,v\rangle\,.\end{equation}
\end{definition}
In other terms, $H(x,p)$ is the support function to the set $F(x)$, evaluated at $p$.

\begin{definition}[Hypothesis (F)]\label{def:hypothesisF}
We say that the hypothesis (F) is satisfied by  system \eqref{eq:diffinc} if the following holds true:
\begin{enumerate}
\item[(F1)] $F(x)$ is nonempty, convex, and compact for each $x\in\mathbb{R}^n$;
\item[(F2)] $F$ is Lipschitz continuous.
\end{enumerate}
\end{definition}

\begin{definition}[Hypothesis (H)]\label{def:hypothesisH}
We say that the hypothesis (H) is satisfied by  system \eqref{eq:diffinc} if the following holds true:
\begin{enumerate}
\item[(H1)] There exists a constant $c_0\geq 0$ such that, for every $p\in\mathbb R^n$,  $x\mapsto H(x,p)$ is semiconvex with semiconvexity constant $c_0|p|$.
\item[(H2)] For all $p\neq 0$, the gradient $\nabla_pH(\cdot,p)$ exists and is globally Lipschitz continuous, i.e. there exists $K_1\ge 0$ such that
\begin{equation}\label{eq:global}
\|\nabla_pH(x,p)-\nabla_pH(y,p)\|\ \leq\ K_1 \|y-x\|,\hspace{1cm}\textrm{for all }x,y\in\mathbb R^n\,, p\in\mathbb R^n\setminus\{0\}.
\end{equation}
\end{enumerate}
\end{definition}

\begin{remark}\label{rem:cons}
If hypothesis (F) hold, and $K$ is a Lipschitz constant for $F$, then for all $v\in F(x_2)$ there exists $\eta\in \mathbb B^n$, $\rho\in[0,1]$, $v'\in F(x_1)$ such that we have
$v=v'+K\rho\eta\|x_1-x_2\|$. Moreover, $H(x,p)\in\mathbb R$ for every $x,p\in\mathbb R^n$.
In particular,
\begin{align*}\langle p,v\rangle&=\langle p,v'\rangle+K\rho\|x_1-x_2\|\langle\eta,p\rangle\\ 
&\le \sup_{w\in F(x_1)}\langle p,w\rangle+K\|p\|\cdot\|x_1-x_2\|=H(x_1,p)+K\|p\|\cdot\|x_1-x_2\|.\end{align*}
By passing to the supremum on $v\in F(x_2)$, we have
\[H(x_2,p)-H(x_1,p)\le K\|p\|\cdot\|x_1-x_2\|,\]
whence, reversing the role of $x_1$ and $x_2$, we end up with
\begin{equation}\label{eq:H-Lipschitz}
|H(x_1,p)-H(x_2,p)| \leq\ K\|p\|\cdot\|x_1-x_2\|,
\end{equation}
for all $x_1,x_2,p\in\mathbb R^n$.
Take now $x_1=x$, $x_2=0$, and let $v\in F(x)$ be such that $v\ne 0$ and $\|v\|=\displaystyle\max_{w\in F(x)}\|w\|$. Set $p=v/\|v\|$.
Then we have 
\[\|v\|=\displaystyle\max_{w\in F(x)}\|w\|=H(x,p)\le K\|x\|+H(0,p)\le K\|x\|+\max_{w\in F(0)}\|w\|.\] 
So, there exists $K_2>0$ such that if $F(x)\ne \{0\}$ it holds:
\begin{equation}\label{extra}
\max_{w\in F(x)}\|w\|\le K_2(1+\|x\|).
\end{equation}
Since the above inequality still holds even if $F(x)=\{0\}$,  it holds for every $x\in\mathbb R^n$.
\end{remark}

\begin{remark}
Global Lipschitz continuity in both (F2) and (H2) was assumed just to simplify computations. Indeed, our results still hold if $F$ is locally Lipschitz with respect to the Hausdorff distance, and $\nabla_pH(\cdot,p)$ is locally Lipschitz in $x$, uniformly so over $p$ in $\mathbb{R}\backslash \{0\}$. In that case, however, we need to assume (\ref{extra}) as an extra condition.
\end{remark}
We collect now, in the following, some consequences of assumptions (F) and (H):

\begin{proposition}\label{prop:consSA}
Suppose that $F$ and $H$ satisfy hypotheses (F) and (H). 
\begin{enumerate}
\item [(i)] For every $p\neq 0$ and $x\in\mathbb{R}^n$, it holds
\[\partial H(x,p)\subseteq \partial_xH(x,p)\times\partial_pH(x,p),\]
where $\partial_x$ and $\partial_p$ denote the Clarke's generalized gradient with respect to $x$ variables and $p$ variables, respectively. 
\item[(ii)] For every $p\neq 0$, if we set $F_p(x)=\nabla_p H(x,p)$, then $F_p(x)\in F(x)$ and 
$\left\langle F_p(x),p\right\rangle=H(x,p)$. 
Moreover, $F_p(\cdot)$ is Lipschitz continuous with Lipschitz constant $K_1$, i.e.,
\[\|F_p(y)-F_p(x)\|\leq K_1\cdot \|y-x\|,\quad\forall x,y\in\mathbb{R}^n.\]  
\item[(iii)] For every $\xi\in\partial_xH(x,p)$, it holds
\[H(y,p)-H(x,p)-\langle\xi,y-x\rangle\geq -c_0\cdot\|p\|\cdot\|y-x\|^2,\quad\forall x,y\in\mathbb{R}^n.\]
\end{enumerate}
\end{proposition}
\begin{proof}
See Proposition 1 in \cite{CaW}.
\end{proof}

\begin{proposition}
Assume that $F$ and $H$ satisfy assumptions (F) and (H). Let $T>0$, $p(\cdot):[0,T]\rightarrow\mathbb{R}^n$ be an 
absolutely continuous arc such that $p(t)\neq 0$ for all $t\in [0,T]$. Given $x_0\in\mathbb{R}^n$, then the problem
\begin{equation}\label{IVP}
\begin{cases}
\dot{x}(t)= F_{p(t)}(x(t)),& \textrm{ for a.e. }t\in [0,T],\vspace{.3cm}\\ x(0)=x_0,
\end{cases}
\end{equation}
admits a unique solution $y(\cdot,x_0)$. Moreover, for all $t>0$, the following holds:
\begin{enumerate}
\item [(i)] $x_0\mapsto y(t,x_0)$ is Lipschitz continuous on $\mathbb{R}^n$ and 
\[\|y(t,z_0)-y(t,x_0)\|\leq e^{Kt}\cdot \|z_0-x_0\|,\quad\forall z_0\in\mathbb{R}^n,\]
\item[(ii)] $\|y(t,x_0)\|\leq (\|x_0\|+1)\cdot e^{K_2t}-1$,
\item[(iii)] $\|y(t,x_0)-x_0\|\leq K_2\cdot(\|x_0\|+1)e^{K_2t}t$,
\end{enumerate}
where $K_2$ is the constant in Remark \ref{rem:cons}.
\end{proposition}
\begin{proof}
The proof is based on Gronwall's inequaliy, see Proposition 3.4, Lemma 3.5, and Lemma 3.7 in \cite{CK}.
\end{proof}

The following statement allows us to make local approximation of trajectories of \eqref{eq:diffinc} with smooth ones:

\begin{proposition}\label{prop:regtraj}
Assume $F$ and $H$ satisfy hypotheses (F) and (H). Let $\mathcal{K}\subset\mathbb{R}^n$ be compact and let $\delta>0$ be given. 
Set
\[M=\sup\left\{\|v\|:\, v\in F (x),\,x\in\mathcal{K}+\delta B(0,1)\right\},\hspace{1cm}\delta_1=\dfrac{\delta}{M+1}.\]
Then for each $x_0\in\mathcal{K}$, and $v\in F(x_0)$ there exists a $C^1$ trajectory $y^{x_0}(\cdot)$ of \eqref{eq:diffinc} such that 
\[\|\dot{y}^x_0(t)-v\|\leq KMt,\quad\forall t\in [0,\delta_1].\]
\end{proposition}
\begin{proof}
See Proposition 2.3 in \cite{CaMW}.
\end{proof}

We recall the following result:

\begin{proposition}\label{prop:dai-0}
Assume that $F$ and $H$ satisfy (F) and (H). Suppose $x(\cdot)$ is an optimal trajectory 
starting from $x_0$ and $\xi\ne 0$ is a unit proximal normal vector to $\mathcal{S}$ at $x(T(x_0))$ so that 
$\mu\doteq H(x(T(x_0)),\xi)^{-1}\leq 0$. 
Then there exists an absolutely continuous arc $p:[0,T]\rightarrow\mathbb{R}^n$ satisfying
\begin{equation}\label{eq:Nsystem}
\begin{cases}\hspace{.3cm}
\dot{x}(t)\in F_{p(t)}(x(t))\vspace{.3cm}\\
-\dot{p}(t)\in \partial_{x}H(x(t),p(t))
\end{cases}
\quad \textrm{ a.e. }t\in [0,T(x_0)]
\end{equation}
and the transversality condition 
\[-p(T)=\mu\cdot\xi.\]
\end{proposition}
\begin{proof}
See Theorem 2.1 in \cite{CaMW}.
\end{proof}

\section{Main results}\label{sec:main}

\subsection{The dual arc inclusion}
%We begin our analysis with a dual arc inclusion that extends and complements the one in \cite{CaMW}.
\begin{theorem}\label{thm:dai} 
Assume that $F$ and $H$ satisfy (F) and (H). Let $x\in\mathbb{R}^n\backslash\mathcal{S}$, let
$x(\cdot)$ be an optimal trajectory which starts from $x$ and reaches $\mathcal{S}$ in time $T(x)$, and 
let $\xi$ be a unit proximal inner normal to $\mathcal{S}$ at $\bar{x}=x(T(x))$. Consider the system:
\begin{equation}\label{eq:Couple-S}
\begin{cases}
\hspace{.3cm}\dot{x}(t)=F_{p(t)}(x(t)),\vspace{.3cm}\\ 
-\dot{p}(t)\in\partial_xH(x(t),p(t))
\end{cases}
\quad \text{for a.e.} \;t\in [0,T(x)].
\end{equation}
Then:
\begin{enumerate}
\item[(i)] if $H(\bar{x},\xi)\ne 0$ and we couple system \eqref{eq:Couple-S} with the condition $p(T(x))=\dfrac{\xi}{H(\bar x,\xi)}$, we have $-p(t)\in\partial^PT(x(t))$ for all $t\in[0,T(x))$,
\item[]
\item[(ii)] if $H(\bar{x},\xi)=0$ and we couple system \eqref{eq:Couple-S} with the condition $p(T(x))=\xi$, we have $-p(t)\in\partial^{\infty}T(x(t))$ for all $t\in [0,T(x))$.
\end{enumerate}
\end{theorem}
\begin{proof} 
Applying the dynamic programming principle, one can see that if the conclusion of the theorem holds for $t=0$ then it will hold for all $t\in [0,T(x))$. 
Therefore, we only need to prove the result for $t=0$. We divide the proof into two main steps.

\medskip

\emph{Step 1 (near the target $\mathcal{S}$)}: There exists a constant $C_1\ge 0$ such that for all $\bar y\in\overline{\mathcal{S}^c}$ and $\beta\leq T(\bar{y})$
it holds:
\begin{equation}\label{eq:Est2}
\langle p(T(x)),\bar{y}-\bar{x}\rangle+\beta H(\bar{x},p(T(x)))\leq C_1\cdot\left(\|\bar{y}-\bar{x}\|^2+\beta^2\right). 
\end{equation}

\medskip

\emph{Proof of Step 1:}
Set $\alpha=H(\bar{x},p(T(x)))$. Since $p(T(x))=\lambda\cdot\xi\in N^P_{\overline{\mathcal{S}^c}}(\bar{x})$, we have that 
there exists $C>0$ such that:
\begin{equation}\label{eq:Est1}
\big\langle p(T(x)),\bar{y}-\bar{x}\big\rangle\leq C\cdot \|\bar{y}-\bar{x}\|^2,\quad\forall \bar{y}\in\overline{\mathcal{S}^c}.
\end{equation}
Therefore, if $\beta\leq 0$, then \eqref{eq:Est2} follows from \eqref{eq:Est1} by choosing $C_1=C$. Otherwise, if $0<\beta\leq T(\bar{y})$, we consider the differential equation
\begin{equation}\label{eq:Forward1}
\begin{cases} \dot{y}(t)&=F_{p(T(x))}(y(t)),\\ y(0)&=\bar{y}, \end{cases}
\end{equation}
and set $\bar{y}_1:=y(\beta)$. Since $\bar{y}_1\in\overline{\mathcal{S}^c}$, recalling \eqref{eq:Est1} we have that
\begin{align*}
\left\langle p(T(x)),\bar{y}_1-\bar{x}\right\rangle&\leq C\cdot \|\bar{y_1}-\bar{x}\|^2
\leq2C\cdot \left(\|\bar{y}_1-\bar{y}\|^2+\|\bar{y}-\bar{x}\|^2\right)\\
&\leq2C\cdot \left(\|y(\beta)-y(0)\|^2+\|\bar{y}-\bar{x}\|^2\right)
\leq C_2\cdot \left(\beta^2+\|\bar{y}-\bar{x}\|^2\right),
\end{align*}
where $C_2$ is a suitable positive constant. On the other hand, we compute 
\begin{align*}
\left\langle p(T(x)),\bar{y}-\bar{y}_1\right\rangle=& \left\langle p(T(x)),-\int^{\beta}_{0}F_{p(T(x))}(y(s))\,ds\right\rangle\\
=&-\int_{0}^{\beta}\left\langle p(T(x)),F_{p(T(x))}(y(s))\right\rangle\,ds\\
=&-\int_{0}^{\beta}\left\langle p(T(x)),F_{p(T(x))}(\bar{x})\right\rangle\,ds+\\
&\hspace{1cm}-\int_{0}^{\beta}\left\langle p(T(x)),F_{p(T(x))}(y(s))-F_{p(T(x))}(\bar{x})\right\rangle\,ds\\
\leq&-\alpha\cdot\beta+C_3\cdot\left(\|\bar{y}-\bar{x}\|^2+\beta^2\right),
\end{align*}
where  $C_3$ is another suitable positive constant (we used also item (ii) in Proposition \ref{prop:consSA}). Combining the two above inequalities, we obtain that
\begin{equation}
\left\langle p(T(x)),\bar{y}-\bar{x}\right\rangle\leq-\alpha\cdot\beta+(C_2+C_3)\cdot\left(\|\bar{y}-\bar{x}\|^2+\beta^2\right).
\end{equation}
Therefore, \eqref{eq:Est2} holds with $C_1=\max\lbrace{C,C_2+C_3\rbrace}$, and this concludes the proof of Step 1.
\hfill$\diamond$

\medskip

\emph{Step 2:} We show that $(p(0),\alpha)\in N^P_{\mathrm{hypo}(T)}(x,T(x))$,
i.e. there exists a positive constant $C_5$ such that for all $y\in\overline{\mathcal{S}^c}$ and $\beta\leq T(y)$:
\begin{equation}\label{eq:Hypo-exterior}
\langle p(0),y-x\rangle+\alpha (\beta-T(x))\leq C_5\cdot \left(\|y-x\|^2+|\beta-T(x)|^2\right)
\end{equation}

\emph{Proof of Step 2: } There are two possible cases:

\begin{figure}[t]
\includegraphics[scale=.7]{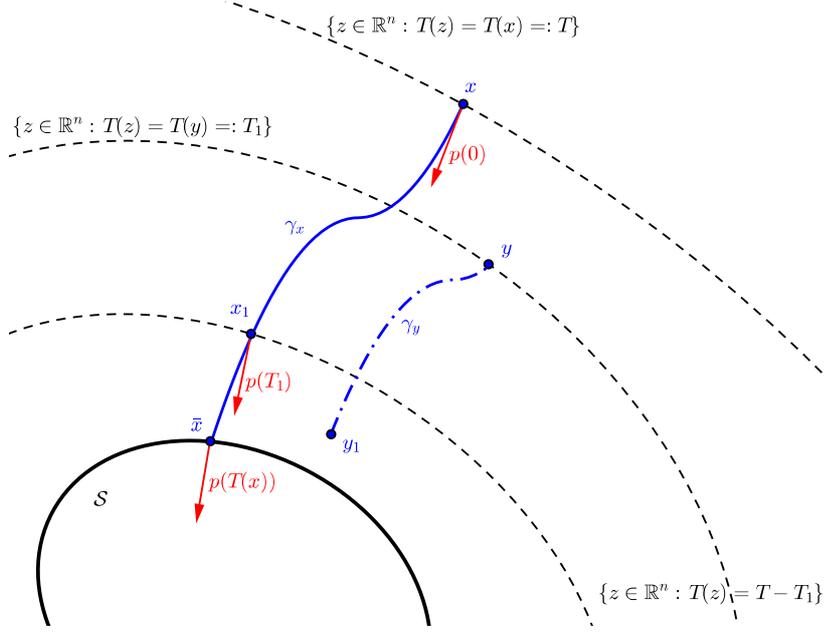}
\caption{First case of the proof of Theorem \ref{thm:dai}, Step 2.}
\end{figure}

\emph{First case:} Suppose $T_1:=T(y)\leq T(x)=:T$. In this case, we only need to show that \eqref{eq:Hypo-exterior} holds for $\beta=T_1$, i.e., 
\begin{equation}\label{eq:case1}
\langle p(0),y-x\rangle+\alpha (T_1-T)\leq C_5\cdot \left(\|y-x\|^2+|T_1-T|^2\right).
\end{equation}
We consider the differential equation:
\begin{equation}\label{eq:Forward2}
\begin{cases}\dot{y}(t)=F_{p(t)}(y(t)),\\ y(0)=y,\end{cases}
\end{equation}
and set $y_1=y(T_1)$, $x_1=x(T_1)$. Observing that $T(x_1)=T-T_1$ and $y_1\in\overline{\mathcal{S}^c}$. We first compute:
\begin{align*}
\langle p(T_1)&,y_1-x_1\big\rangle=\langle p(0),y-x\rangle+\int_{0}^{T_1}\frac{d}{ds}\langle p(s),y(s)-x(s)\rangle ds\,\\
&=\langle p(0),y-x\rangle+\int_{0}^{T_1}\langle\dot{p}(s),y(s)-x(s)\rangle+\langle p(s),F_{p(s)}(y(s))-F_{p(s)}(x(s))\rangle\, ds\\
&=\langle p(0),y-x\rangle+\int_{0}^{T_1}\langle\dot{p}(s),y(s)-x(s)\rangle+H(y(s),p(s))-H(x(s),p(s))ds.
\end{align*}
Since $-\dot{p}(s)\in\partial_xH(x(s),p(s))$, by Proposition \ref{prop:consSA} we have that
\[\langle -\dot{p}(s),y(s)-x(s)\rangle+H(x(s),p(s))-H(y(s),p(s))\leq c_0\cdot \|p(s)\|\cdot \|y(s)-x(s)\|^2.\]
Hence, for a suitable constant $C_6>0$, it holds
\[\langle -\dot{p}(s),y(s)-x(s)\rangle+H(x(s),p(s))-H(y(s),p(s))\leq C_6\cdot \|y-x\|^2\]
Therefore, 
\begin{equation}\label{eq:Est6}
\langle p(0),y-x\rangle\leq \left\langle p(T_1),y_1-x_1\right\rangle+ C_6\cdot \|y-x\|^2.
\end{equation}
Now, we compute
\begin{align*}
 \left\langle p(T_1),y_1-x_1\right\rangle=& \left\langle p(T),y_1-x_1\right\rangle+ \left\langle p(T_1)-p(T),y_1-x_1\right\rangle\\
 \leq& \left\langle p(T),y_1-x_1\right\rangle+C_7\cdot(|T-T_1|^2+\|y-x\|^2),
\end{align*}
and 
\begin{align*}
 \big\langle p(T),&y_1-x_1\big\rangle= \left\langle p(T),x(T)-x(T_1)\right\rangle+ \left\langle p(T),y_1-\bar{x}\right\rangle\\
 \leq&\int_{T_1}^T\left\langle p(T),F_{p(s)}(x(s))\right\rangle ds+\left\langle p(T),y_1-\bar{x}\right\rangle\\
 =&\int_{T_1}^T\left\langle p(T),F_{p(s)}(x(s))-F_{p(T)}(x(T))\right\rangle ds+\alpha\cdot (T-T_1)+\left\langle p(T),y_1-\bar{x}\right\rangle.
\end{align*}
Recalling \eqref{eq:Est2}, we have that
\[\left\langle p(T),y_1-\bar{x}\right\rangle\leq C\cdot \|y_1-\bar{x}\|^2\leq C_8\cdot \left(|T-T_1|^2+\|y_1-\bar x\|^2\right),\]
for a suitable positive constant $C_8$. On the other hand, for any $s\in [T_1,T]$,
\begin{align*}
\left\langle p(T),F_{p(s)}(x(s))-F_{p(T)}(x(T))\right\rangle=\hspace{-15em}&\\
&=\left\langle p(T),F_{p(T)}(x(s))-F_{p(T)}(x(T))\right\rangle+\left\langle p(T),F_{p(s)}(x(s))-F_{p(T)}(x(s))\right\rangle\\
&\leq C_9\cdot|s-T|,
\end{align*}
for some constant $C_9>0$.
Hence,
\[\int_{T_1}^T\left\langle p(T),F_{p(s)}(x(s))-F_{p(T)}(x(T))\right\rangle ds\leq\frac{C_9}{2}\cdot |T-T_1|^2.\]
Therefore, we obtain that
\[\left\langle p(T),y_1-x_1\right\rangle\leq\alpha\cdot (T-T_1)+C_{10}\cdot\left(|T-T_1|^2+\|y-x\|^2\right),\]
for a suitable  constant $C_{10}$.
Combining this estimate with \eqref{eq:Est6}, we finally get \eqref{eq:case1}.

\medskip

\emph{Second case:} $T_1:=T(y)>T(x)=:T$. We consider the differential equation
\begin{equation}\label{eq:Forward3}
\begin{cases}\dot{y}(t)=F_{p(t)}(y(t)),\\ y(0)=y, \end{cases}
\end{equation}
and set $\bar{y}=y(T)$. By  the same arguments of the first step, we have that
\begin{equation}
\langle p(0),y-x\rangle\leq\langle p(T),\bar{y}-\bar{x}\rangle +C\cdot \|y-x\|^2.
\end{equation}
On the other hand, observing that $\bar{y}\in\mathcal{S}^c$,  from \eqref{eq:Est2} we obtain that
\[\left\langle p(T),\bar{y}-\bar{x}\right\rangle+\alpha\cdot\beta_1\leq C_1\cdot\left(\|\bar{y}-\bar{x}\|^2+\beta_1^2\right)\]
for all $\beta_1\leq T(\bar{y})$. Thus, for every $\beta\leq T(y)$, by the dynamic programming principle we have that $\beta-T\leq T(\bar{y})$. 
Hence, taking $\beta_1=\beta-T$ in the above inequality, we get that
\[\left\langle p(T),\bar{y}-\bar{x}\right\rangle+\alpha\cdot(\beta-T)\leq C_1\cdot\left(\|\bar{y}-\bar{x}\|^2+|\beta-T|^2\right).\] 
By Gronwall's inequality, we finally obtain that there exists $C_{11}>0$ such that
\[\left\langle p(T),\bar{y}-\bar{x}\right\rangle+\alpha\cdot(\beta-T)\leq C_{11}\cdot\left(\|y-x\|^2+|\beta-T|^2\right).\]
The proof is complete.
\end{proof}

\subsection{The constancy of $H$}
\begin{proposition}\label{pr:H}
Assume (F) and (H). Let $x\in\mathbb{R}^n\backslash\mathcal{S}$, let $x(\cdot)$ be an optimal trajectory at $x$  
reaching $\mathcal{S}$ in time $T(x)$, and let $\xi$ be a  proximal inner unit normal to $\mathcal{S}$ at $\bar{x}=x(T(x))$. Consider the system:
\begin{equation}\label{eq:Couple-S1}
\begin{cases}\hspace{.3cm}\dot{x}(t)=F_{p(t)}(x(t)),\vspace{.2cm}\\ -\dot{p}(t)\in\partial_xH(x(t),p(t))
 \end{cases}
 \quad \text{for a.e.} \;t\in [0,T(x)].
\end{equation}
Then:
\begin{enumerate}
\item[(i)] if $H(\bar{x},\xi)\ne 0$ and we couple system \eqref{eq:Couple-S1} with the condition $p(T(x))=\dfrac{\xi}{H(\bar x,\xi)}$, we have $H(x(t),p(t))=1$ for all $t\in[0,T(x)]$,
\item[]
\item[(ii)] if $H(\bar{x},\xi)=0$ and we couple system \eqref{eq:Couple-S1} with the condition $p(T(x))=\xi$, we have $H(x(t),p(t))=0$ for all $t\in [0,T(x)]$.
\end{enumerate}
\end{proposition}
\begin{proof}
\begin{enumerate}
\item[]
\item[(i)] Assume that $H(\bar{x},p(T(x)))=1$. Fixing $t\in (0,T(x))$ we will show that $H(x(t),p(t))=1$. 
First, observe that, by Theorem~\ref{thm:dai}, $-p(t)\in\partial^PT(x(t))$ for all $t\in [0,T(x))$. 
Now, for every $v\in F(x(t))$,  Proposition \ref{prop:regtraj} ensures that there exists a trajectory $y(\cdot)\in C^1$ starting from $x(t)$ at time $t$ such that
\[\|\dot{y}(t+s)-v\|\leq K\cdot s,\quad\forall s\in(0,\delta).\] 
Since $-p(t)\in\partial^PT(x(t))$, we have that
\begin{align*}
\left\langle p(t),y(t+s)-x(t)\right\rangle+\beta-T(x(t))\le&\\ 
&\hspace{-10em}\leq C\cdot \left(\|y(t+s)-x(t)\|^2+|\beta-T(x(t))|^2\right),
\end{align*}
for all $\beta\leq T(y(t+s))$. Observe that $T(y(t+s))\geq T(x(t))-s$ and choose $\beta=T(x(t))-s$ in the above inequality to obtain 
\[\left\langle p(t),y(t+s)-x(t)\right\rangle-s\leq C_1\cdot s^2.\]
Hence,
\[\langle p(t),v\rangle=\lim_{s\rightarrow 0^+}\frac{\left\langle p(t),y(t+s)-x(t)\right\rangle}{s}\leq 1.\] 
This implies that $H(x(t),p(t))\leq 1$.

On the other hand, we have
\begin{align*}
\langle p(t),x(t-s)-x(t)\rangle+T(x(t-s))-T(x(t))\le&\\ 
&\hspace{-15em}\le C\cdot \left(\|x(t-s)-x(t)\|^2+|T(x(t-s))-T(x(t))|^2\right).
\end{align*}
Since $T(x(t-s))-T(x(t))=s$, we get that
\[\langle p(t),x(t-s)-x(t)\rangle+s\leq C_1\cdot s^2.\]
thus
\[1\leq\limsup_{s\rightarrow 0^+}\left\langle p(t),\frac{x(t)-x(t-s)}{s}\right\rangle\leq H(x(t),p(t)).\]
Therefore, we finally obtain that 
\[H(x(t),p(t))=1,\quad\forall t\in (0,T(x)].\]
By the continuity of $H$, the proof of (i) is complete.
\item[]
\item[(ii)] Assume that $H(\bar{x},p(T(x)))=0$. Fixing $t\in (0,T(x))$, we will show that $H(x(t),p(t))=0$.
For every $v\in F(x(t))$, there exists a trajectory $y(\cdot)\in C^1$ starting from $x(t)$ at time $t$ such that
\[\|\dot{y}(t+s)-v\|\leq K\cdot s,\quad\forall s\in(0,\delta).\] 
Since $-p(t)\in\partial^{\infty}T(x(t))$, we have that
\[\left\langle p(t),y(t+s)-x(t)\right\rangle\leq C\cdot\left(\|y(t+s)-x(t)\|^2+|\beta-T(x(t))|^2\right),\]
for every $\beta\leq T(y(t+s))$. Since $T(y(t+s))\geq T(x(t))-s$, we can choose $\beta=T(x(t))-s$ to obtain, by the above inequality,
\[\left\langle p(t),y(t+s)-x(t)\right\rangle\leq C_1\cdot s^2.\]
Hence 
\[\langle p(t),v\rangle=\lim_{s\rightarrow 0^+}\left\langle p(t),\frac{y(t+s)-x(t)}{s}\right\rangle\leq 0.\]
Therefore, $H(x(t),p(t))\leq 0$.

On the other hand,  
\begin{align*}\langle p(t),x(t-s)-x(t)\rangle&\leq\\ &\hspace{-5em}\le C\cdot \left(\|x(t-s)-x(t)\|^2+|T(x(t-s))-T(x(t))|^2\right).\end{align*}
Since $T(x(t-s))-T(x(t))=s$,  from the above inequality we deduce that 
\[\langle p(t),x(t-s)-x(t)\rangle\leq C_1\cdot s^2.\]
Thus,
\[0\leq\limsup_{s\rightarrow 0^+}\left\langle p(t),\frac{x(t)-x(t-s)}{s}\right\rangle\leq H(x(t),p(t)).\]
Therefore, 
\[H(x(t),p(t))=0,\quad\forall t\in (0,T(x)].\]
By the continuity of $H$, the proof of (ii) is complete.
\end{enumerate}
\end{proof}

\subsection{Optimality conditions}

\begin{theorem} 
Assume (F) and (H). Let $x\in\mathbb{R}^n\setminus\mathcal{S}$ and let $x(\cdot)$ 
be a trajectory starting from $x$ such that $x(T)\in\mathcal{S}$. Suppose  there exists a 
continuous function $p(\cdot):[0,T]\to\mathbb{R}^n$ such that  
\begin{equation}\label{eq:Sufficient-C}
\begin{cases}
\hspace{.3cm}\dot{x}(t)=F_{p(t)}(x(t))
\vspace{.2cm}\\
-p(t)\in\partial^{P}T(x(t))\vspace{.2cm}\\
\hspace{.3cm}H(x(t),p(t))=1
\end{cases}
\quad \text{for a.e.}\; t\in[0,T]\,.
\end{equation}
Then $x(\cdot)$ is an optimal trajectory.  
\end{theorem}
\begin{proof}
Set $\varphi(\cdot)=T(x(\cdot))$. Owing to the dynamic programming principle,
for every $t\in[0,T]$ and $0<s<T-t$ we have
\[\varphi(t+s)\geq\varphi(t)-s.\]
Hence, $\displaystyle\liminf_{s\rightarrow 0^+}\dfrac{\varphi(t+s)-\varphi(t)}{s}\geq -1$. 

On the other hand, since $-p(t)\in\partial^{P}T(x(t))$, we have that 
\begin{align*}
\langle p(t),x(t+s)-x(t)\rangle+\varphi(t+s)-\varphi(t)&\leq\\ &\hspace{-10em}\leq C\cdot\left(\|x(t+s)-x(t)\|^2+|\varphi(t+s)-\varphi(t)|^2\right).
\end{align*}
This is equivalent to
{\small
\begin{align*}
C\cdot&\frac{\|x(t+s)-x(t)\|^2}{s}\geq\\ 
&\geq\left\langle p(t),\frac{x(t+s)-x(t)}{s}\right\rangle+\frac{\varphi(t+s)-\varphi(t)}{s}\cdot\left(1-C\cdot(\varphi(t+s)-\varphi(t))\right)\\
&=\frac{1}{s}\int_0^s \langle p(t),F_{p(t+\tau)}(x(t+\tau))\rangle\,d\tau+\frac{\varphi(t+s)-\varphi(t)}{s}\cdot\left(1-C\cdot(\varphi(t+s)-\varphi(t))\right)\\
&=\frac{1}{s}\int_0^s \langle p(t+\tau),F_{p(t+\tau)}(x(t+\tau))\rangle\,d\tau+\frac{1}{s}\int_0^s \langle p(t)-p(t+\tau),F_{p(t+\tau)}(x(t+\tau))\rangle\,d\tau+\\
&\hspace{1cm}+\frac{\varphi(t+s)-\varphi(t)}{s}\cdot\left(1-C\cdot(\varphi(t+s)-\varphi(t))\right)\\
&=\frac{1}{s}\int_0^s H(x(t+\tau),p(t))\,d\tau+\frac{1}{s}\int_0^s \langle p(t)-p(t+\tau),F_{p(t+\tau)}(x(t+\tau))\rangle\,d\tau+\\
&\hspace{1cm}+\frac{\varphi(t+s)-\varphi(t)}{s}\cdot\left(1-C\cdot(\varphi(t+s)-\varphi(t))\right),\\
&=1+\frac{1}{s}\int_0^s \langle p(t)-p(t+\tau),F_{p(t+\tau)}(x(t+\tau))\rangle\,d\tau+\\
&\hspace{1cm}+\frac{\varphi(t+s)-\varphi(t)}{s}\cdot\left(1-C\cdot(\varphi(t+s)-\varphi(t))\right),
\end{align*}}
where we used the Mean Value Theorem and the fact that 
\[\langle p(t+\tau),F_{p(t+\tau)}(x(t+\tau))\rangle=H(p(t+\tau),x(t+\tau))=1,\] for a.e. $\tau [0,s]$ and for every $0<s<T-t$.
The continuity of $p(\cdot)$ and the estimate
\[\lim_{s\rightarrow 0^+}\dfrac{\|x(t+s)-x(t)\|^2}{s}=0,\]
now yield for every $t\in]0,T[$:
\[\limsup_{s\rightarrow 0^+}\dfrac{\varphi(t+s)-\varphi(t)}{s}\leq -1.\]
Thus, $\displaystyle\lim_{s\rightarrow 0^+}\dfrac{\varphi(t+s)-\varphi(t)}{s}=-1$ for every $0<t<T$.

Now, for $s<0$, we have that $\varphi(t+s)\leq \varphi(t)-s$. Thus, $\dfrac{\varphi(t+s)-\varphi(t)}{s}\geq -1$ for all $0<t<T$ and $-t<s<0$. Hence,
\[\liminf_{s\rightarrow 0^-}\frac{\varphi(t+s)-\varphi(t)}{s}\geq -1.\]
On the other hand, we have
\begin{align*}
C\cdot&\|x(t)-x(t+s)\|^2\geq\\
&\ge\left\langle p(t+s),x(t)-x(t+s)\right\rangle+(\varphi(t)-\varphi(t+s))\cdot(1-C\cdot\left(\varphi(t)-\varphi(t+s))\right).
\end{align*}
For $s<0$, by performing similar computations as done in the case $0<s<T-t$, from the above inequality we get that
{\small
\begin{align*}
&C\cdot\frac{\|x(t)-x(t+s)\|^2}s\geq\\
&\geq\left\langle p(t+s),\frac{x(t)-x(t+s)}{s}\right\rangle+\frac{\varphi(t)-\varphi(t+s)}{s}\cdot(1-C\cdot\left(\varphi(t)-\varphi(t+s))\right)\\
&\geq\left\langle p(t+s)-p(t),\frac{x(t)-x(t+s)}{s}\right\rangle+\left\langle p(t),\frac{x(t)-x(t+s)}{s}\right\rangle+\\
&+\frac{\varphi(t)-\varphi(t+s)}{s}\cdot(1-C\cdot\left(\varphi(t)-\varphi(t+s))\right)\\
&\geq\left\langle p(t+s)-p(t),\frac{x(t)-x(t+s)}{s}\right\rangle-\dfrac 1s\int_0^s\left\langle p(t),F_{p(t+\tau)}(x(t+\tau))\right\rangle\,d\tau+\\
&+\frac{\varphi(t)-\varphi(t+s)}{s}\cdot(1-C\cdot\left(\varphi(t)-\varphi(t+s))\right)\\
&\geq\left\langle p(t+s)-p(t),\frac{x(t)-x(t+s)}{s}\right\rangle-\dfrac 1s\int_0^s\left\langle p(t)-p(t+\tau),F_{p(t+\tau)}(x(t+\tau))\right\rangle\,d\tau+\\
&+\frac{\varphi(t)-\varphi(t+s)}{s}\cdot(1-C\cdot\left(\varphi(t)-\varphi(t+s))\right)-\dfrac 1s\int_0^s\left\langle p(t+\tau),F_{p(t+\tau)}(x(t+\tau))\right\rangle\,d\tau.
\end{align*}}
Thus, for every $0<t<T$ we have
\[\limsup_{s\rightarrow 0^-}\frac{\varphi(t+s)-\varphi(t)}{s}\leq -1.\]
This implies that $\lim_{s\rightarrow 0^-}\frac{\varphi(t+s)-\varphi(t)}{s}=-1$. 
Therefore, we finally have that $\dot{\varphi}(t)=-1$ for all $t\in ]0,T(x)[$, so
\[T(x(t))=\varphi(t)=\varphi(0)-t=T(x)-t,\]
that is, $x(\cdot)$ is an optimal trajectory.
\end{proof}

\subsection{Regularity results}$ $

\bigskip

In order to prove the regularity results of this section we need the following technical lemma, the proof of which is postponed to the Appendix.

\begin{lemma}\label{lemma:techno}
Let $D\subseteq\mathbb R^n$ be open,  let $f\in C^0(D)$ be such that $\mathrm{hypo}\,f$ satisfies an external sphere condition with a locally uniform radius,
and let $V\subseteq D$ be a bounded open  set. If $\nabla f$ is essentially bounded on $V$, then $f$ is Lipschitz continuous on $V$. 
\end{lemma}

\begin{lemma}\label{lemma:petr_revis}
Let $\mathcal S\subseteq \mathbb R^n$ be a closed nonempty set with a $C^{1,1}$-smooth  boundary. Assume $(F)$ and $(H)$. Suppose  $H\in C^{1,1}\left(\mathbb R^n\times(\mathbb R^n\setminus\{0\})\right)$
and  $T(\cdot)$ is continuous in $\mathcal R$.
Then $x_0\in\mathcal R\setminus\mathcal S$ is a point at which $T(\cdot)$ fails to be Lipschitz continuous if and only if there exists:
\begin{itemize}
\item[(i)]  an optimal trajectory $x_0(\cdot)$
with $x_0(0)=x_0$, $\bar x_0:=x_0(T(x_0))\in\mathcal S$, and
\item[(ii)] a unit vector  $\xi_0\in N^{P}_{\overline{\mathbb R^n\setminus \mathcal S}}(\bar x_0)$  such that $H(\bar x_0,\xi_0)=0$. 
\end{itemize}
\end{lemma}
%Consequently, if (i) and (ii) above hold true, then $\partial^{\infty}T(x_0)\ne\varnothing$.
\begin{proof} First of all, observe that, if (i) and (ii)  hold true at a point  $x_0\in\mathcal R\setminus\mathcal S$, then 
$\partial^{\infty}T(x_0)\ne\varnothing$ thanks to Theorem~\ref{thm:dai}. Therefore, $T(\cdot)$ cannot be Lipschitz at $x_0$.

\medskip
In order to prove the converse, let $x_0\in\mathcal R\setminus\mathcal S$ be a point at which $T(\cdot)$ fails to be Lipschitz. Note that, owing to Theorem~4.1 in  \cite{CK}, $\mathrm{hypo}\,T$ satisfies an external sphere condition with a locally constant radius. Consequently, by Lemma~4.1 and Corollary~4.1 in \cite{NK}, $x\mapsto N^P_{\mathrm{hypo}\,T}(x,T(x))$ has closed graph, 
and $\partial^PT(x)$ reduces to the singleton $\{\nabla T(x)\}$ at every $x\in\mathrm{dom}\,\nabla T$. Moreover, $T(\cdot)$ is differentiable a.e. 
in $\mathcal R\setminus S$.
%
%\medskip
%
%Now, let $x_0\in\mathcal R\setminus\mathcal S$ be a point at which $T(\cdot)$ fails to be Lipschitz continuous.  
Then, by  Lemma~\ref{lemma:techno}, $\nabla T$ must be unbounded on any neighborhood of $x_0$. So, there exists a sequence $\{x_j\}_{j\geqslant 1}\subset\mathrm{dom}\,\nabla T$ and a unit vector $\zeta_0\in \mathbb R^n$ such that 
\[
 \lim_{j\to\infty} x_j= x_0\,,\quad \lim_{j\to\infty}|\nabla T(x_j)|=\infty \quad\text{and}\quad \lim_{j\to\infty}\dfrac{\left(-\nabla T(x_j),1\right)}{|\left(-\nabla T(x_j),1\right)|}=(\zeta_0,0).\]
Notice that, since the map  $x\mapsto N^P_{\mathrm{hypo}\,T}(x,T(x))$ has closed graph, $\zeta_0\in \partial^{\infty} T(x_0)$. 

%\medskip

Next, thanks to Proposition~\ref{prop:dai-0}, we can construct sequences of absolutely continuous arcs $\{x_j(\cdot)\}_{\geqslant 1}$, $\{p_j(\cdot)\}_{\geqslant 1}$,  and unit vectors $\{\xi_j\}_{\geqslant 1}$
with the following properties:
\begin{enumerate}
\item[(a)] $x_j(0)=x_j$ and $\bar x_j:=x_j(T(x_j))\in \mathcal S$, i.e., $x_j(\cdot)$ is an optimal trajectory starting from $x_j$ which reaches the target at $\bar x_j$;
\item[(b)] $\xi_j\in N^P_{\overline{\mathbb R^n\setminus \mathcal S}}(\bar x_j)$, and $H(\bar x_j,\xi_j)\ne 0$;
\item[(c)] $p_j:=p_j(T(x_j))=\dfrac{\xi_j}{H(\bar x_j,\xi_j)}$;
\item[(d)] $(x_j(\cdot),p_j(\cdot))$ satisfies
\begin{equation}\label{eq:Couple-S3}
\begin{cases}
\hspace{.3cm}\dot{x_j}(t)=\nabla_pH(x_j(t),p_j(t))\vspace{.2cm}\\ 
-\dot{p_j}(t)=\nabla_xH(x_j(t),p_j(t))
\end{cases}
\quad \textrm{for a.e.} \;t\in [0,T(x_j)].
\end{equation}
\end{enumerate}
Also, since $x_j\in\mathrm{dom}\,\nabla T$,   Theorem~\ref{thm:dai} ensures that  $-p_j(0)=\nabla T(x_j)$.

%\medskip

Since $x_j\to x_0$ and $T(\cdot)$ is continuous,  a well known compactness property of the trajectories of the differential inclusion  
\begin{equation}\label{eq:limit}
\dot x(t)\in F(x(t))
\end{equation}
ensure that, up to extracting a subsequence, one can assume that $\{x_j(\cdot)\}_{j\in\geqslant 1}$ converges uniformly to an absolutely continuous arc, $x_0(\cdot)$, which satisfies \eqref{eq:limit} together 
with $x_0(0)=x_0$. Again, since $T(\cdot)$ is continuous and $\mathcal S$ is closed, we have that $\bar x_0:=x_0(T(x_0))\in \mathcal S$, i.e. $x_0(\cdot)$ is optimal.
We can also assume  that $\{\xi_j\}_{j\in\mathbb N}$ converges to the (proximal)  inner unit normal $\xi_0$ to $\mathcal S$ at $\bar x_0$ by the smoothness  of $\mathcal S$.

%\medskip
We now claim that $|p_j|\to \infty$ as $j\to\infty$. Indeed, since $H$ is homogeneous of degree $1$ in $p$ and of class $C^{1,1}\left(\mathbb R^n\times(\mathbb R^n\setminus\{0\})\right)$, for some constant $M>0$
\[|\nabla_x H(x,p)|\leq  M|p|\]
for all $p\in R^n\setminus\{0\}$ and  all $x$ in the bounded set
\[A:=\big\{x_j(t)~:~t\in [0,T(x_j)]\big\}\cup \big\{x_0(t)~:~t\in[0,T(x_0)]\big\}.\]
Therefore,  \eqref{eq:Couple-S3} yields, for all $j\geq  1$ and a.e. $t\in [0,T(x_j)]$,
\[-\dfrac{1}{2}\dfrac{d}{dt}|p_j(t)|^2=\langle -\dot p_j(t),p_j(t)\rangle=\big\langle \nabla_x H\big(x_j(t),p_j(t)\big),p_j(t)\big\rangle\leq  M|p_j(t)|^2.\]
The above differential inequality implies that
\begin{equation*}
|p_j|\geq  e^{-MT(x_j)}|p_j(0)|.
\end{equation*}
Since $|p_j(0)|\to \infty$ and $T(x_j)\to T(x_0)$ as $j\to\infty$, we conclude that also $|p_j|\to \infty$.

%\medskip
Finally, observe that  $p_j=|p_j|\cdot\xi_j$ and $H(\bar x_j,p_j)=1$ for every $j\in\mathbb N$. So,
\[1=\lim_{j\to\infty}H(\bar x_j,p_j)=\lim_{j\to\infty}|p_j|\cdot H(\bar x_j,\xi_j).\]
Since $|p_j|\to \infty$, we must have that 
\begin{equation*}
\lim_{j\to \infty}H(\bar x_j,\xi_j)=H(\bar x_0,\xi_0)=0.
\end{equation*}
This shows points (i) and (ii) above. Invoke Theorem~\ref{thm:dai} to complete the proof .
\end{proof}

\begin{theorem}\label{thm:sbv}
Suppose $H\in C^{1,1}\big(\mathbb R^n\times(\mathbb R^n\setminus\{0\})\big)$, and let 
 $g\in C^2_b(\mathbb R^n)$. Define the target set $\mathcal S$ by  \[\mathcal S=\{x\in\mathbb R^n~:~g(x)\le 0\},\] 
and let $\Sigma$ be a countably $\mathscr H^{n-2}$-rectifiable subset of $\partial\mathcal S$.
In addition to $(F)$ and $(H)$, assume that:
\begin{enumerate}
\item[(a)] $0\in F(x)$ for every $x\in\partial\mathcal S$,
\item[(b)] $\nabla g\ne 0$ on $\partial\mathcal S$,
\item[(c)] $T(\cdot)$ is continuous in $\mathcal R$, and
\item[(d)] for all $x\in\partial \mathcal S\setminus\Sigma$ and $\lambda\in\mathbb R$
\begin{equation}\label{eq:nonzero_2}
\nabla_x H(x,-\nabla g(x))-\nabla_p H(x,-\nabla g(x))\cdot \nabla^2 g(x)\ne \lambda\nabla g(x)\,.
\end{equation}
\end{enumerate}
Then:
\begin{enumerate}
\item[(i)] $T(\cdot)$ is a function of special bounded variation on $\mathcal R\setminus \mathcal S$, and
\item[(ii)] there exists a closed countably 
$\mathscr H^{n-1}$-rectifiable set 
$\mathscr S\subset \mathcal R$ such that $T(\cdot)$ is locally semiconcave on  $\mathcal R\setminus \big(\mathcal S\cup\mathscr S\big)$.
%such that $\mathcal R\setminus \mathcal O$ \footnote{Consequently,  $\mathscr L^n(\mathcal R\setminus \mathcal O)=0$.}.
\end{enumerate}
\end{theorem}
\begin{proof}
We begin by showing the following property needed for the sequel of the proof: 
\begin{itemize}
\item[$(\star)$] if, for some point   $\bar x\in\partial\mathcal S$, there exists a nonzero vector $\xi\in N^P_{\overline{\mathbb R^n\setminus\mathcal S}}(x)$ such that $H(\bar x,\xi)=0$, then
$\bar x\in \Sigma$.
\end{itemize}
Indeed, since $0\in F(x)$, we have that $H(x,p)\ge 0$ for all $x\in\partial\mathcal S$ and $p\in\mathbb R^n$. 
Therefore, $H$ attains a constrained minimum at $(\bar x,\xi)$. Moreover, since $H(\bar x,\cdot)$ is positively homogeneous, 
$(\bar x,r \xi)$ also minimizes $H$ constrained to $\partial\mathcal S$ for all $r \geq  0$. 
Now,  $-\nabla g(\bar x)$ can be represented as $r \xi$ for some $r>0$ because $\mathcal S$ has a smooth boundary. 
So, $\bar x$ is a minimizer of the smooth function $x\mapsto H(x,-\nabla g(x))$ on $\partial \mathcal S$. Then, by Lagrange multiplier rule and \eqref{eq:nonzero_2}, we have $\bar x\in \Sigma$.

\medskip
Next, let us prove the $SBV$ property~(i). Since  $\mathrm{hypo}~T$ satisfies an external sphere condition, owing to Proposition~6.1 in \cite{MKV} we have that $T\in BV_{loc}(\mathcal R\backslash\mathcal{S})$. So, we just need to show that, on $\mathcal{R}\backslash\mathcal{S}$, the singular part of $DT(\cdot)$ is concentrated on a set of $\sigma$-finite $\mathscr H^{n-1}$-measure.  Recalling that the characteristics of  the pde 
\begin{equation*}
H(x,-\nabla T)-1=0
\end{equation*}
satisfy
\begin{equation}\label{eq:char}
\begin{cases}
\hspace{.cm}
\dot X(t)=-\nabla_p H(X(t),P(t))
&
\\
\hspace{.cm}
\dot P(t)=\nabla_x H(X(t),P(t)),
\end{cases}
\end{equation}
consider the flow $(X,P):\mathbb R\times\mathbb R^n\times\mathbb R^n\to \mathbb R^n\times\mathbb R^n$, that is, for any $(x,p)\in\mathbb R^n\times\mathbb R^n$, $(X,P)(\cdot,x,p)$ is the solution pf \eqref{eq:char} with initial conditions
\begin{equation}\label{eq:charic}
X(0)=x\,,\quad P(0)=p.
\end{equation}
Notice that $(t,x)\mapsto (X,P)\big(t,x,-\nabla g(x)\big)$ is Lipschitz continuous under our assumptions. Now, set
\begin{equation*}
\Phi(t,x)=X\big(t,x,-\nabla g(x)\big)\qquad t\geq 0,\;x\in \mathbb R^n.
\end{equation*}
Then, by Lemma~\ref{lemma:petr_revis} we have that, 
%if
%$T(\cdot)$ fails to be Lipschitz at a point $x\in\mathcal R\backslash\mathcal{S}$, then  there exists $T_x>0$ and $\bar x\in\Sigma$ such that
%%\begin{equation*}
%$x=\Phi(T_x,\bar x)$.
%%\end{equation*}
%If we define, 
for any $t>0$,
\[\Sigma_t:=\big\{x\in\mathcal R\setminus S~:\partial^{\infty} T(x)\ne\varnothing\,,\; T(x)\leq t\big\}
\subseteq\Phi\big((0,t]\times\Sigma\big)\]
%Thus,  $\Sigma_t$ 
%equals $\phi\big((0,t]\times\Sigma\big)$, and 
is countably $\mathscr H^{n-1}$-rectifiable since $\Sigma$ is countably $\mathscr H^{n-2}$-rectifiable by hypothesis.  Thus,  
\[\mathscr S:=\big\{x\in\mathcal R\setminus \mathcal S~:~\textrm{$T(\cdot)$ fails to be Lipschitz at $x$}\big\}
=\bigcup_{k=1}^\infty\Sigma_k \]
%So, $\mathscr S$ 
is countably $\mathscr H^{n-1}$-rectifiable as well.  Moreover, owing to Lemma~4.2 and Corollary~4.2 in \cite{NK},   $\mathscr S$ is closed in $\mathcal R$. Therefore, in view of Proposition \ref{propSBV}, we conclude that $T(\cdot)\in SBV_{loc}(\mathcal R\setminus\mathcal S)$.
%\\
%\quad\\
 
Finally, by Corollary~4.3 in \cite{CK}, $T(\cdot)$ is  locally semiconcave on $\mathcal R\setminus\big(\mathcal S\cup\mathscr S\big)$ since it is locally Lipschitz  on such a set. 
 \end{proof}

\begin{example}[Prescribing singularities]
Let  $g\in C^2_b(\mathbb R^n)$, let
\[\mathcal S:=\{x\in\mathbb R^n~:~g(x)\le 0\},\]
and assume that $\nabla g(x)\ne 0$ at every $x\in\mathcal S$.
Choose a countable set $\Sigma\subseteq \mathcal{S}$ such that, for some function
 $\psi:\Omega\to[0,+\infty[$ of class $C^{1,1}$, we have $\Sigma=\{y\in\mathbb R^n~:~\psi(y)=0\}$.
 Define
\[H(x,p):=\psi(x)\mathrm{dist}\big(p,\nabla g(x)^\perp\big),\]
where $\nabla g(x)^\perp$ denotes the hyperplane  in $\mathbb R^n$ that is orthogonal to $\nabla g(x)$.
For  all $x,p,q\in\mathbb R^n$, we have that $H(x,\lambda p)=\lambda H(x,p)\ge 0$ for every $\lambda\ge 0$, and $H(x,p+q)\le H(x,p)+H(x,q)$.  So,
$H(x,\cdot)$ is the support function of a closed convex subset $F(x)$ of $\mathbb R^n$.
With these choices, we have that \eqref{eq:nonzero_2} is satisfied.
\end{example}

\appendix

\section{Proof of  Lemma \ref{lemma:techno}}

We now prove  the technical Lemma \ref{lemma:techno}.

\begin{proof}
%(of Lemma \ref{lemma:techno}). 
Since  $\mathrm{hypo}\,f$ satisfies the external sphere condition, for some $\theta>0$, we have that for every $y,z\in V$ 
there exists a unit vector $v\in N^P_{\mathrm{hypo}\,f}(z,f(z))$ such that
\begin{equation}\label{eq:esp}\left\langle v,\left(y-z,f(y)-f(z)\right)\right\rangle\le \dfrac{1}{2\theta}\left(|y-z|^2+|f(y)-f(z)|^2\right).\end{equation}
For $f(y)\ne f(z)$, dividing the above inequality by $|f(y)-f(z)|$ we obtain
\begin{align}\label{eq:esp2}
\left\langle v,\left(\dfrac{y-z}{|f(y)-f(z)|},\dfrac{f(y)-f(z)}{|f(y)-f(z)|}\right)\right\rangle&\le \dfrac{1}{2\theta}\left(\dfrac{|y-z|}{\frac{|f(y)-f(z)|}{|y-z|}}+|f(y)-f(z)|\right).
\end{align}
By contradiction, suppose there exists $L>0$ such that $|\nabla f(q)|\le L$ for a.e. $q\in V$ and  there exists $\bar x\in V$ such that
\begin{equation}\label{eq:appecontra}
\limsup_{\substack{y,x\to \bar x\\ y\ne x}}\dfrac{|f(y)-f(x)|}{|y-x|}=\infty.
\end{equation}
Let $\{y_j\}_{j\in\mathbb N}\subseteq V$ and $\{x_j\}_{j\in\mathbb N}\subseteq V$ be sequences converging to $\bar x$ such that $x_j\ne y_j$ for all $j\in\mathbb N$, and
\[\lim_{j\to \infty}\dfrac{|f(y_j)-f(x_j)|}{|y_j-x_j|}=\infty.\]

We claim that there exist $\{z_j\}_{j\in\mathbb N}$ and $\{w_j\}_{j\in\mathbb N}$ such that $z_j\to \bar x$, $w_j\to\bar x$, $f$ is differentiable at $z_j$ and at $w_j$, $|\nabla f(z_j)|\le L$ 
and $|\nabla f(w_j)|\le L$ for every $j\in\mathbb N$,
and
\[
\lim_{j\to \infty}\dfrac{|f(z_j)-f(w_j)|}{|z_j-w_j|}=\infty.
\]
Indeed, recalling that $f\in C^0(V)$, for every $j\in\mathbb N$ there exists $\delta_j>0$ such that if $|y_j-q|\le \delta_j$, we have $|f(y_j)-f(q)|\le |f(y_j)-f(x_j)|^2$. 
Since $f$ is a.e. differentiable on $V$ by Theorem~3.1 in \cite{NK}, and $|\nabla f(q)|\le L$ for a.e. $q\in V$ by assumption,
for every $j\in\mathbb N$ we can find 
$z_j\in B(y_j,\delta_j)\cap B\left(y_j,|y_j-x_j|^2\right)$ such that $f$ is differentiable at $z_j$,
$|\nabla f(z_j)|\le L$, and $z_j\ne x_j$. 
Notice that
\begin{align*}
|x_j-y_j|&\le |x_j-z_j|+|z_j-y_j|\le|x_j-z_j|+|x_j-y_j|^2\\
|x_j-z_j|&\le |x_j-y_j|+|y_j-z_j|\le |x_j-y_j|+|x_j-y_j|^2.
\end{align*}
Dividing both inequalities by $|x_j-y_j|$ and letting $j\to \infty$ yields
\[\lim_{j\to \infty}\dfrac{|x_j-z_j|}{|x_j-y_j|}=1.\]
Similarly, since
\begin{align*}
|f(x_j)-f(y_j)|&\le |f(x_j)-f(z_j)|+|f(z_j)-f(y_j)|\\&\le|f(x_j)-f(z_j)|+|f(x_j)-f(y_j)|^2\\ &\\
|f(x_j)-f(z_j)|&\le |f(x_j)-f(y_j)|+|f(y_j)-f(z_j)|\\&\le |f(x_j)-f(y_j)|+|f(x_j)-f(y_j)|^2,
\end{align*}
we obtain
\[\lim_{j\to \infty}\dfrac{|f(x_j)-f(z_j)|}{|f(x_j)-f(y_j)|}=1.\]
So,
\begin{align*}\lim_{j\to \infty}\dfrac{|f(z_j)-f(x_j)|}{|z_j-x_j|}&=
\lim_{j\to\infty}\dfrac{|f(z_j)-f(x_j)|}{|f(y_j)-f(x_j)|}\dfrac{|y_j-x_j|}{|z_j-x_j|}\dfrac{|f(y_j)-f(x_j)|}{|y_j-x_j|}\\ &=\lim_{j\to \infty}\dfrac{|f(y_j)-f(x_j)|}{|y_j-x_j|}=\infty.\end{align*}
To construct $w_j$, it suffices to apply the above reasoning to $x_j,z_j$ instead of $y_j,x_j$.

\medskip

Next, recall that, owing to Lemma~4.1 in \cite{NK},
\begin{align*}
N^P_{\mathrm{hypo}\,f}(z_j,f(z_j))=\left\{\lambda \dfrac{(-\nabla f(z_j),1)}{|(-\nabla f(z_j),1)|}:\,\lambda\ge 0\right\},\\ 
N^P_{\mathrm{hypo}\,f}(w_j,f(w_j))=\left\{\lambda \dfrac{(-\nabla f(w_j),1)}{|(-\nabla f(w_j),1)|}:\,\lambda\ge 0\right\}. 
\end{align*}
We can assume that, as $j\to +\infty$,
\begin{align*}
\dfrac{(-\nabla f(z_j),1)}{|(-\nabla f(z_j),1)|}\to (\zeta,\xi)\;\; \text{and}\;\;\dfrac{(-\nabla f(w_j),1)}{|(-\nabla f(w_j),1)|}\to (\eta,\sigma)\;\;\text{with}\;\; \xi,\sigma\ge\dfrac{1}{\sqrt{1+C^2}}>0
\end{align*}
since $\nabla f(z_j)$ and $\nabla f(w_j)$ are bounded. 
Moreover,  $(\zeta,\xi)$, $(\eta,\sigma)\in N^P_{\mathrm{hypo}\,f}(\bar x,f(\bar x))$ (since $\theta$ does not depend on $z_j$ or $w_j$), and, for every $y\in V$, we have that
\begin{align*}
\left\langle(\zeta,\xi),\left(y-\bar x,f(y)-f(\bar x)\right)\right\rangle&\le \dfrac{1}{2\theta}\left(|y-\bar x|^2+|f(y)-f(\bar x)|^2\right).\\
\end{align*}
We now distinguish  two cases.
\begin{enumerate}
\item[a.] Assume that $\{f(z_j)-f(x_j)\}_{j\in\mathbb N}$ has a monotone increasing subsequence. In this case, since it tends to $0$, after relabeling, we have $f(x_j)-f(z_j)\ge 0$ for every $j\in\mathbb N$. 
Choosing, in \eqref{eq:esp2}, 
\begin{equation*}
y=x_j,\quad z=z_j,\quad v=\dfrac{(-\nabla f(z_j),1)}{|(-\nabla f(z_j),1)|},
\end{equation*}
we have
\begin{align*}
\left\langle\dfrac{(-\nabla f(z_j),1)}{|(-\nabla f(z_j),1)|},\left(\dfrac{x_j-z_j}{|f(x_j)-f(z_j)|},\dfrac{f(x_j)-f(z_j)}{|f(x_j)-f(z_j)|}\right)\right\rangle\le\\
&\hspace{-5cm}\le \dfrac{1}{2\theta}\left(\dfrac{|x_j-z_j|}{\frac{|f(x_j)-f(z_j)|}{|x_j-z_j|}}+|f(x_j)-f(z_j)|\right).
\end{align*}
Passing to the limit in the above inequality yields $\xi\le 0$, which is a contradiction with the fact that $\xi>0$.
\item[]
\item[b.] Assume that $\{f(z_j)-f(x_j)\}_{j\in\mathbb N}$ has a monotone decreasing subsequence. In this case, we have that $\{f(x_j)-f(z_j)\}_{j\in\mathbb N}$ has
a monotone increasing subsequence, and we apply the argument of the previous case by switching the roles of $x_j$ and $z_j$, and replacing $\xi$ with $\sigma$.
\end{enumerate}
Since we have reached a contradiction assuming \eqref{eq:appecontra},  the proof is complete.
\end{proof}


\begin{bibdiv}
\begin{biblist}

\bib{ACK}{article}{
author={Albano, Paolo},
author={Cannarsa, Piermarco},
author={Nguyen, Khai T.},
author={Sinestrari, Carlo},
title={Singular gradient flow of the distance function and homotopy equivalence},
journal={Math. Ann.},
volume={356},
date={2013},
number={1},
pages={23--43},
}

\bib{AC}{book}{
   author={Aubin, Jean-Pierre},
   author={Cellina, Arrigo},
   title={Differential inclusions},
   series={Grundlehren der Mathematischen Wissenschaften [Fundamental
   Principles of Mathematical Sciences]},
   volume={264},
   note={Set-valued maps and viability theory},
   publisher={Springer-Verlag},
   place={Berlin},
   date={1984},
   pages={xiii+342},
   isbn={3-540-13105-1},
   review={\MR{755330 (85j:49010)}},
}

\bib{AF}{book}{
   author={Aubin, Jean-Pierre},
   author={Frankowska, Halina},
   title={Set-Valued Analysis},
   publisher={Birkh\"auser},
   place={Boston},
   date={1990},
}

\bib{AFP}{book}{
   author={Ambrosio, Luigi},
   author={Fusco, Nicola},
   author={Pallara, Diego},
   title={Functions of bounded variation and free discontinuity problems},
   series={Oxford Mathematical Monographs},
   publisher={The Clarendon Press Oxford University Press},
   place={New York},
   date={2000},
   pages={xviii+434},
   isbn={0-19-850245-1},
   review={\MR{1857292 (2003a:49002)}},
}
% 
% \bib{BCD}{book}{
%    author={Bardi, Martino},
%    author={Capuzzo-Dolcetta, Italo},
%    title={Optimal control and viscosity solutions of Hamilton-Jacobi-Bellman
%    equations},
%    series={Systems \& Control: Foundations \& Applications},
%    note={With appendices by Maurizio Falcone and Pierpaolo Soravia},
%    publisher={Birkh\"auser Boston Inc.},
%    place={Boston, MA},
%    date={1997},
%    pages={xviii+570},
%    isbn={0-8176-3640-4},
%    review={\MR{1484411 (99e:49001)}},
%    doi={10.1007/978-0-8176-4755-1},
% }

\bib{CK}{article}{
   author={Cannarsa, Piermarco},
   author={Nguyen, Khai T.},
   title={Exterior sphere condition and time optimal control for differential inclusions},
   journal={SIAM J. Control Optim.},
   volume={49},
   date={2011},
   number={6},
   pages={2258--2576},
}

\bib{CS0}{article}{
   author={Cannarsa, Piermarco},
   author={Sinestrari, Carlo},
   title={Convexity properties of the minimum time function},
   journal={Calc. Var. Partial Differential Equations},
   volume={3},
   date={1995},
   number={3},
   pages={273--298},
   issn={0944-2669},
   review={\MR{1385289 (97f:49032)}},
   doi={10.1007/s005260050016},
}

\bib{CS}{book}{
   author={Cannarsa, Piermarco},
   author={Sinestrari, Carlo},
   title={Semiconcave functions, Hamilton-Jacobi equations, and optimal control},
   series={Progress in Nonlinear Differential Equations and their
   Applications, 58},
   publisher={Birkh\"auser Boston Inc.},
   place={Boston, MA},
   date={2004},
   pages={xiv+304},
   isbn={0-8176-4084-3},
   review={\MR{2041617 (2005e:49001)}},
}

\bib{CaW}{article}{
   author={Cannarsa, Piermarco},
   author={Wolenski, Peter R.},
   title={Semiconcavity of the value function for a class of differential inclusions},
   journal={Discrete Contin. Dyn. Syst.},
   volume={29},
   date={2011},
   number={2},
   pages={453--466},
   issn={1078-0947},
   review={\MR{2728465 (2012a:49046)}},
   doi={10.3934/dcds.2011.29.453},
}

\bib{CaMW}{article}{
   author={Cannarsa, Piermarco},
   author={Marino, Francesco},
   author={Wolenski, Peter},
   title={The dual arc inclusion with differential inclusions},
   journal={Nonlinear Anal.},
   volume={79},
   date={2013},
   pages={176--189},
   issn={0362-546X},
   review={\MR{3005035}},
   doi={10.1016/j.na.2012.11.021},
}

% \bib{FHC}{book}{
%    author={Clarke, Francis H.},
%    title={Optimization and nonsmooth analysis},
%    series={Classics in Applied Mathematics},
%    volume={5},
%    edition={2},
%    publisher={Society for Industrial and Applied Mathematics (SIAM)},
%    place={Philadelphia, PA},
%    date={1990},
%    pages={xii+308},
%    isbn={0-89871-256-4},
%    review={\MR{1058436 (91e:49001)}},
% }

% \bib{CLSW}{book}{
%    author={Clarke, Francis H.},
%    author={Ledyaev, Yuri S.},
%    author={Stern, Ronald J.},
%    author={Wolenski, Peter R.},
%    title={Nonsmooth analysis and control theory},
%    series={Graduate Texts in Mathematics},
%    volume={178},
%    publisher={Springer-Verlag},
%    place={New York},
%    date={1998},
%    pages={xiv+276},
%    isbn={0-387-98336-8},
%    review={\MR{1488695 (99a:49001)}},
% }

\bib{CM}{article}{
   author={Colombo, Giovanni},
   author={Marigonda, Antonio},
   title={Differentiability properties for a class of non-convex functions},
   journal={Calc. Var. Partial Differential Equations},
   volume={25},
   date={2006},
   number={1},
   pages={1--31},
   issn={0944-2669},
   review={\MR{2183853 (2006h:49027)}},
   doi={10.1007/s00526-005-0352-7},
}

\bib{CM3}{article}{
   author={Colombo, Giovanni},
   author={Marigonda, Antonio},
   title={Singularities for a class of non-convex sets and functions, and
   viscosity solutions of some Hamilton-Jacobi equations},
   journal={J. Convex Anal.},
   volume={15},
   date={2008},
   number={1},
   pages={105--129},
   issn={0944-6532},
   review={\MR{2389006 (2009i:49023)}},
}

\bib{CMW}{article}{
   author={Colombo, Giovanni},
   author={Marigonda, Antonio},
   author={Wolenski, Peter R.},
   title={Some new regularity properties for the minimal time function},
   journal={SIAM J. Control Optim.},
   volume={44},
   date={2006},
   number={6},
   pages={2285--2299 (electronic)},
   issn={0363-0129},
   review={\MR{2248184 (2008d:49021)}},
   doi={10.1137/050630076},
}

\bib{CMW2}{article}{
   author={Colombo, Giovanni},
   author={Marigonda, Antonio},
   author={Wolenski, Peter R.},
   title={The Clarke generalized gradient for functions whose epigraph has positive reach},
   journal={Math. Op. Res.},
   volume={38},
   date={2013},
   number={3},
   pages={451--468},
   issn={0364-765X},
   doi={10.1287/moor.1120.0580},
}

\bib{CKH}{article}{
   author={Colombo, Giovanni},
   author={Nguyen, Khai T.},
   title={On the structure of the minimum time function},
   journal={SIAM J. Control Optim.},
   volume={48},
   date={2010},
   number={7},
   pages={4776--4814},
   issn={0363-0129},
   review={\MR{2720234}},
   doi={10.1137/090774549},
}

\bib{CTh}{book}{
   author={Colombo, Giovanni},
   author={Lionel Thibault}, 
   title={Prox-regular sets and applications, \textup{in} Handbook of Nonconvex Analysis, D.Y. Gao, D. Motreanu eds.},
   publisher={International Press},
   place={Somerville, MA},
   date={2010},
   pages={680},
   isbn={978-1-57146-200-8},
}

\bib{CVL}{article}{
   author={Colombo, Giovanni},
   author={Nguyen, Khai T.}, 
   author={Nguyen, Luong V.}, 
   title={SBV regularity for minimum time function},
   journal={Calc. Var. Partial Differential Equations},
   date={2013},
   doi={10.1007/s00526-013-0682-9},
}

\bib{MKV}{article}{
author={Marigonda, Antonio}, 
author={Nguyen, Khai T.}, 
author={Vittone, Davide},
title={Some regularity properties for a class of upper semicontinuous functions}, 
journal={Indiana Univ. Math. Journal.},
volume={62},
date={2013},
number={1}
}

\bib{NK}{article}{
   author={Nguyen, Khai T.},
   title={Hypographs satisfying an external sphere condition and the regularity of the minimum time function},
   journal={J. Math. Anal. Appl.},
   volume={372},
   date={2010},
   %number={},
   pages={611--628},
   %issn={},
   %review={},
   doi={10.1016/j.jmaa.2010.07.010},
}

\bib{NV}{article}{
   author={Nguyen, Khai T.},
   author={Vittone, Davide},
   title={Rectifiability of siguralities of non-lipschitz functions},
   journal={J. of Convex Analysis},
   volume={19},
   date={2012},
   number={1},
   pages={159--170},
   %issn={},
   %review={},
   %doi={},
}

% \bib{O}{article}{
% author={Ornelas, A.},
% title={Parametrization of {C}arath\'eodory multifunctions},
% journal={Rend. Sem. Mat. Univ. Padova}, 
% volume={83},
% date={1990}, 
% pages={33--44},
% }
% 
% \bib{PR}{article}{
%    author={Poliquin, R. A.},
%    author={Rockafellar, R. Tyrrell},
%    title={Prox-regular functions in variational analysis},
%    journal={Trans. Amer. Math. Soc.},
%    volume={348},
%    date={1996},
%    number={5},
%    pages={1805--1838},
%    issn={0002-9947},
%    review={\MR{1333397 (96h:49039)}},
%    doi={10.1090/S0002-9947-96-01544-9},
% }
% 
% 
% \bib{RTR2}{article}{
%    author={Rockafellar, R. Tyrrell},
%    title={Clarke's tangent cones and the boundaries of closed sets in ${\bf
%    R}^{n}$},
%    journal={Nonlinear Anal.},
%    volume={3},
%    date={1979},
%    number={1},
%    pages={145--154 (1978)},
%    issn={0362-546X},
%    review={\MR{520481 (80d:49032)}},
%    doi={10.1016/0362-546X(79)90044-0},
% }
% 
% \bib{Ro}{book}{
%    author={Rockafellar, R. Tyrrell},
%    title={Convex analysis},
%    series={Princeton Mathematical Series, No. 28},
%    publisher={Princeton University Press},
%    place={Princeton, N.J.},
%    date={1970},
%    pages={xviii+451},
%    review={\MR{0274683 (43 \#445)}},
% }
% 
% \bib{RW}{book}{
%    author={Rockafellar, R. Tyrrell},
%    author={Wets, Roger J.-B.},
%    title={Variational analysis},
%    series={Grundlehren der Mathematischen Wissenschaften [Fundamental
%    Principles of Mathematical Sciences]},
%    volume={317},
%    publisher={Springer-Verlag},
%    place={Berlin},
%    date={1998},
%    pages={xiv+733},
%    isbn={3-540-62772-3},
%    review={\MR{1491362 (98m:49001)}},
%    doi={10.1007/978-3-642-02431-3},
% }

\end{biblist}
\end{bibdiv}
\end{document}